\documentclass[11pt]{amsart}
\usepackage[margin=1.25in]{geometry}

\usepackage{comment}
\usepackage{mathrsfs}
\usepackage{amsfonts}
\usepackage{latexsym,epsfig}
\usepackage{amsmath, amsthm, amscd, amssymb, centernot}
\usepackage[pdftex]{color}
\usepackage{comment}
\usepackage{marginnote}
\usepackage[colorlinks,citecolor=cyan,linkcolor=teal]{hyperref}
\usepackage[nameinlink]{cleveref}
\usepackage{enumitem}
\usepackage[normalem]{ulem}

\usepackage{eso-pic}

\usepackage{tikz}
\usepackage{tikz-cd}
\usetikzlibrary{matrix,decorations.pathmorphing,arrows}
\tikzset{commutative diagrams/.cd,arrow style=tikz,diagrams={>=stealth'}}
\usepackage{todonotes}

\newcommand{\HH}[0]{\mathbb{H}}

\newcommand{\Z}{\mathbb{Z}}

\newcommand{\R}{\mathbb{R}}

\renewcommand{\l}{\lambda}

\newcommand{\wt}{\widetilde}
\newcommand{\mc}{\mathcal}

\newcommand{\wh}{\widehat}

\newcommand{\N}{\mathbb{N}}

\newcommand{\orb}{\mathcal{O}}

\newcommand\tsim{\kern-.4em\sim}

\newcommand\ssm{\smallsetminus}

\newcommand{\spn}{\mathrm{span}}

\newcommand{\eplus}{\mathbf{e}^+}
\newcommand{\emin}{\mathbf{e}^-}
\newcommand{\epm}{\mathbf{e}^\pm}

\renewcommand{\phi}{\varphi}
\renewcommand{\epsilon}{\varepsilon}

\newcommand{\defn}[1]{\textbf{\emph{#1}}}

\DeclareMathOperator{\core}{core}

\newcommand{\univ}{S^1_{\mathrm{univ}}}

\newcommand{\lu}{\mathrm{LU}}
\newcommand{\rd}{\mathrm{RD}}
\newcommand{\Ends}{\mathrm{Ends}}
\newcommand{\ct}{\mathrm{CT}}

\newcommand{\snm}{s_\mathbf{nm}}
\newcommand{\vep}{\varepsilon}

\DeclareMathOperator*{\limset}{\mathrm{lim}\; \mathrm{set}}
\DeclareMathOperator*{\acset}{\mathrm{acc}}

\usepackage{hyperref}

\usepackage{thmtools}

\numberwithin{equation}{section}

\newtheorem{theorem}[equation]{Theorem}

\newtheorem{thm}{Theorem}

\newtheorem{cor}[thm]{Corollary}

\newtheorem{lemma}[equation]{Lemma}
\newtheorem{proposition}[equation]{Proposition}
\newtheorem{corollary}[equation]{Corollary}

\theoremstyle{definition}
\newtheorem{question}[equation]{Question}

\newtheorem{example}[equation]{Example}

\declaretheorem[style=definition,qed=$\lozenge$,sibling=theorem]{remark}
\declaretheorem[style=definition,qed=$\lozenge$,sibling=theorem]{definition}

\newtheorem{convention}[equation]{Convention}
\newtheorem*{theorem*}{Theorem}

\begin{document}
\title{Cannon--Thurston maps for Anosov foliations}

\author{Ellis Buckminster}
\address{Department of Mathematics\\
	University of Pennsylvania}
\email{\href{mailto:ellis17@sas.upenn.edu}{ellis17@sas.upenn.edu}}

\date{\today}

\begin{abstract}

	Universal circles, introduced by Thurston and Calegari--Dunfield, are not well understood in general. Recently, the author together with Taylor showed that Anosov foliations with branching admit nonconjugate universal circles. We continue the study of these universal circles and show that for an Anosov foliation with branching on a hyperbolic manifold, the leftmost universal circle admits a Cannon--Thurston-type map to the ideal 2-sphere. This is a new type of construction of a Cannon--Thurston map. 
	As a corollary, we show the fundamental group of the manifold acts on the leftmost universal circle with pseudo-Anosov dynamics.

\end{abstract}
\maketitle

\setcounter{tocdepth}{1}
\tableofcontents

\section{Introduction}\label{sec:intro}

Given a closed, fibered hyperbolic 3-manifold $M$, let $\Sigma$ denote a fiber surface. The lift $\wt\Sigma$ of $\Sigma$ to the universal cover $\wt M$ of $M$ can be identified with the hyperbolic plane $\HH^2$. Thus, $\wt\Sigma$ can be compactified with its \emph{ideal circle boundary} $\partial_\infty\Sigma\cong\partial_\infty\HH^2$. This circle admits an action of $\pi_1(M)$ preserving a pair of laminations and is naturally identified with the ideal circle boundary of any other lifted fiber. The monodromy is a pseudo-Anosov homeomorphism on $\Sigma$, and the orbit space $\orb$ of the suspension flow of this map is naturally homeomorphic to $\wt\Sigma$. Furthermore, the ideal circle boundary $\partial\orb$ of $\orb$ is naturally homeomorphic to $\partial_\infty\wt\Sigma$. See \Cref{fig:overview} for an overview of how these structures fit together.

\begin{figure}[h]
	\centering
	\includegraphics[width=0.8\linewidth]{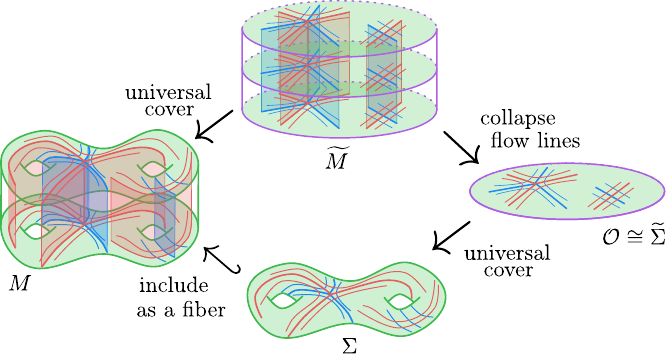}
	\caption{How the different structures fit together in the fibered case.}\label{fig:overview}
\end{figure}

\emph{Universal circles}, defined by Thurston \cite{thurston_circles2} and studied further by Calegari--Dunfield \cite{CalDun_UC}, are an attempt to generalize this picture to a broader class of actions of 3-manifold groups on circles coming from taut foliations. In general, there is much we do not understand about the various possible universal circles for a given foliation, or when two constructions of universal circles result in nonconjugate actions on circles. 

Some rigidity results have been obtained by Calegari for $\R$-covered foliations \cite{Calegari2000Rcovered} and for foliations with one-sided branching \cite{calegari2003foliations}, and more recently by Huang for depth-one foliations transverse to pseudo-Anosov flows without perfect fits \cite{Huang_UC}. 

On the other hand, the author together with Taylor recently showed that for the weak stable or unstable foliations of non $\R$-covered Anosov flows, two constructions of universal circles, due to Calegari--Dunfield \cite{CalDun_UC} and Fenley and Landry--Minsky--Taylor \cite{fenley2012ideal,LMT_UC}, result in nonconjugate actions on circles \cite{buckminster2025universal}. Thus, this class of foliations seems like a good testing ground for various questions about universal circles.

\medskip

Returning to the fibered case, Cannon--Thurston showed that the inclusion of $\wt\Sigma$ into $\wt M$ extends continuously to a $\pi_1(M)$-equivariant map $\partial_\infty\wt\Sigma\to S^2_\infty$, where $S^2_\infty$ is the ideal 2-sphere boundary of $\wt M\cong\HH^3$ \cite{CannonThurston}. As the action $\pi_1(M)\curvearrowright S^2_\infty$ is minimal, the map $\partial_\infty\wt\Sigma\to S^2_\infty$ is necessarily surjective. This map can be obtained by collapsing the $\pi_1(M)$-invariant prelaminations on $\partial_\infty\wt\Sigma$. More generally, Frankel showed that the ideal circle boundary of a \emph{quasigeodesic flow} on a hyperbolic 3-manifold $M$ admits a continuous, $\pi_1(M)$-equivariant map to the ideal 2-sphere boundary of $\wt M$, and the map produced by Cannon--Thurston is a special case of this result \cite{Frankel_thesis}.

Thus, it is natural to ask the following.

\begin{question}\label{q:main}
	When does a universal circle for a foliation on a closed, hyperbolic 3-manifold $M$ admit a Cannon--Thurston-type map, i.e. a continuous, $\pi_1(M)$-equivariant map to $S^2_\infty$?
\end{question}

This is the question we address here, in the case of the \emph{leftmost universal circle} for the weak unstable foliation of a non $\R$-covered Anosov flow.

\subsection*{Results}
Throughout, $\mc C^\ell$ denotes the leftmost universal circle of a non $\R$-covered Anosov foliation in a closed, hyperbolic 3-manifold $M$, and $S^2_\infty$ is the 2-sphere boundary of $\wt M\cong\HH^3$ (see \Cref{sec:background} for definitions). We focus on non $\R$-covered Anosov flows because the case of $\R$-covered Anosov flows is essentially already known; see below. Our main result is the following, which says that the leftmost universal circle $\mc C^\ell$ admits a Cannon--Thurston-type map.

\begin{thm}\label{thma:ct}
	There is a continuous, surjective, $\pi_1(M)$-equivariant map
	\[
	\ct\colon\mc C^\ell\to S^2_\infty.
	\]

\end{thm}

As the actions of $\pi_1(M)$ on $\mc C^\ell$ and $\partial\orb$ (the boundary of the orbit space for the Anosov flow) are nonconjugate by \cite[Theorem B]{buckminster2025universal}, the map $\ct$ is different from the Cannon--Thurston map $e\colon\partial\orb\to S^2_\infty$ due to Frankel \cite{Frankel_thesis} and Fenley \cite{fenley2022nonrcoveredanosovflows}.

Forthcoming work of Fenley--Mann--Potrie \cite{FMP} shows that any Cannon--Thurston map is uniformly finite-to-one. Combining this fact with the dynamics of $\pi_1(M)\curvearrowright S^2_\infty$, we obtain the following corollary.

\begin{cor}\label{cor:dyn}
	Every element of $\pi_1(M)$ has some power acting on $\mc C^\ell$ with a positive, finite number of fixed points, which alternate between attractors and repellors.
\end{cor}

\medskip

In general, given a foliation with the continuous extension property, we could ask when the limit sets of the different leaves can be `stitched together' into a single $\pi_1(M)$-equivariant sphere-filling curve. To make this precise, we make the following definition.

\begin{definition}\label{def:ct_for_fol}
	Let $\mc F$ be a taut, coorientable foliation of a closed, hyperbolic 3-manifold $M$ whose leaves have the continuous extension property. Let 
	\[
	i_\l\colon\partial_\infty\l\to S^2_\infty
	\] 
	denote the continuous extension of a leaf $\l\in\mc L$. A \defn{Cannon--Thurston map} for $\mc F$ is a minimal universal circle $(\mc C,\{\pi_\l\}_{\l\in\mc L})$ together with a continuous, $\pi_1(M)$-equivariant map 
	\[
	h\colon\mc C\to S^2_\infty
	\]
	such that 
	\[
	h(s)=i_\l\circ\pi_\l(s)
	\]
	whenever $s\in\mathrm{core}(\pi_\l)$.
\end{definition}

Note that Cannon--Thurston maps coming from fibered 3-manifolds satisfy this definition where $\mc F$ is the foliation by fibers and $\mc C$ is $\partial\orb$, the boundary of the flow space for the suspension flow. More generally, Landry--Taylor show in forthcoming work \cite{LT} that for any foliation $\mc F$ (almost) transverse to a quasigeodesic pseudo-Anosov flow, the map $e\colon\partial\orb\to S^2_\infty$ gives a Cannon--Thurston map for $\mc F$ in the sense of \Cref{def:ct_for_fol}.

Intuitively, this definition says that the sphere-filling curve resulting from the Cannon--Thurston map stitches together the continuous extensions of each leaf in a way that is compatible with the universal circle structure. We show in \Cref{thm:ct} that the map $\ct$ satisfies this definition.

Given an appropriate foliation $\mc F$ and a universal circle $\mc C$ for $\mc F$, \Cref{def:ct_for_fol} can be viewed as an attempt to \emph{define} a Cannon--Thurston map 
\[
h\colon\mc C\to S^2_\infty
\]
as follows. Given a point $s\in \mc C$, if $s\in\core(\pi_l)$, then define $h(s)=i_l\circ\pi_l(s)$. The issues that then arise are the following. The map is not well-defined, as we may have $s\in\core(\pi_\mu)$ for some other leaf $\mu$. Furthermore, not every point of $\mc C$ is guaranteed to be in some core, so we have not defined the map at every point. Finally, it is not clear that the map is continuous, on whatever domain on which it is defined.

\subsection*{$\R$-covered case}
For the weak unstable foliation of an $\R$-covered Anosov flow $\phi$ on a hyperbolic 3-manifold, the rightmost universal circle arises as the flow space boundary $\partial\orb_\psi$ of a \emph{regulating pseudo-Anosov flow} $\psi$, constructed by Calegari \cite[Corollary 5.3.16]{Calegari2000Rcovered} and Fenley \cite{fenley2012ideal} (see also \cite[Section 6]{buckminster2025universal}). This flow has no perfect fits \cite[Theorem G]{fenley2012ideal} and therefore is quasigeodesic \cite[Main Theorem]{Fen16}. Thus, the map $e\colon\partial\orb_\psi\to S^2_\infty$ constructed by Frankel \cite{frankel2015quasigeodesic} defines a Cannon--Thurston map for the leftmost universal circle. This is a Cannon--Thurston map in the sense of \Cref{def:ct_for_fol} by \cite{LT}.

\subsection*{Proof strategy}
The idea behind the construction of the map $\ct$ is this. The Anosov flow $\varphi$ is quasigeodesic by work of Fenley \cite{fenley2022nonrcoveredanosovflows}. Thus, the inclusion $\l\to \wt M$ of each leaf of the lifted unstable foliation $\wt W^u$ extends continuously to a map $i_\l\colon\partial_\infty\l\to S^2_\infty$ on the boundary at infinity of the leaf $\l$. Morally, for each section $s$ of $\mc C^\ell$ we define the Cannon--Thurston map at $s$ by picking a point $s(\l)\in\partial_\infty\l$ for some leaf  $\l$ that `looks like' a basepoint for that section. We then map that point to $S^2_\infty$ using the map $i_\l$. 

Two main issues arise. First, the choice of basepoint is not canonical. Second, we have to verify continuity of the Cannon--Thurston map. The maps $i_\l\colon\partial_\infty\l\to S^2_\infty$ do not vary continuously as you move through the leaf space. Furthermore, for a sequence of sections $\{s_j\}_{j\in\N}$ converging to a section $s$, the basepoints of the $s_j$'s might not converge to a basepoint of $s$, or might fail to converge at all. 

The way we address both of these issues is by studying a type of extremal behavior that sections of $\mc C^\ell$ can exhibit over paths in the leaf space. One place this behavior occurs is over the set of `basepoints' of a section, which is used in proving well-definedness. It is also related to the limiting behavior of sections, which is important in the proof of continuity. Via the \emph{stitching map} $\Phi\colon\orb\to E_\infty$ defined in \cite{buckminster2025universal}, extremal portions of a section are the images of subsets of \emph{master sets} in the orbit space, and therefore have constant images under the maps $i_\l$ by work of Frankel in \cite{frankel_closedorbits}.

\subsection*{Outline of paper}
In \Cref{sec:background} we go over the necessary background on Anosov flows and universal circles. In \Cref{sec:i} we present a few different viewpoints on the maps $i_\l$ and extend these maps to the end space of the leaf space $\mc L^u$. In \Cref{sec:basepoint} we extend the notion of a basepoint of a special section to all sections of $\mc C^\ell$. This is done by studying the \emph{leftmost up} and \emph{rightmost down} regions in $\mc L^u$ for a given section. We show that every section of $\mc C^\ell$ is either special, or is `based at' a unique end of $\mc L^u$. In \Cref{sec:qle} we define the Cannon--Thurston map and introduce the notion of \emph{quadrant-local extremality} for leftmost sections. This is the idea that relates the leftmost up and rightmost down behavior of sections under limits with the structure of master sets in the orbit space, allowing us to prove that the Cannon--Thurston map is well-defined and continuous in \Cref{sec:thm_proof}.

\subsection*{Future directions}
We end this section by discussing some questions raised by \Cref{thma:ct}.

Landry--Minsky--Taylor conjectured that as actions on circles, all universal circles for non $\R$-covered foliations arise as flowspace boundaries of transverse almost pseudo-Anosov flows \cite[Conjecture 1.5]{LMT_UC}. If this is true in our setting, it would mirror the picture of \emph{regulating flows} for skew Anosov foliations developed by Calegari \cite{Calegari2000Rcovered} and Fenley \cite{fenley2002foliations}. They show that for $W^u$ the weak unstable foliation of a skew Anosov flow, the action on the leftmost universal circle is conjugate to the action on the flowspace boundary of a pseudo-Anosov flow transverse to $\psi$. 

One could also ask the following, which takes into account the universal circle structure of $\mc C^\ell$.

\begin{question}\label{q:from_a_flow_uc}
	Is $\mc C^\ell$ isomorphic \emph{as a universal circle} to the flowspace boundary of some almost transverse pseudo-Anosov flow?
\end{question}

A positive answer to \Cref{q:from_a_flow_uc} would require showing that $\mc C^\ell$ is equivariantly homeomorphic to $\partial\orb_\psi$ for some pseudo-Anosov flow $\psi$, and showing that this homeomorphism intertwines the monotone maps coming from the universal circle structure (see \Cref{def:uc}). \Cref{thma:ct} could be viewed as evidence that \cite[Conjecture 1.5]{LMT_UC} and \Cref{q:from_a_flow_uc} should have positive answers.

\subsection*{Acknowledgments}
I would like to thank Danny Calegari, Lucas Kerbs, and Katie Mann for helpful discussions during this project, and Michael Landry for comments on an earlier draft. I would especially like to thank my advisor, Sam Taylor, for encouraging me, for sending me down this path, and for having many conversations with me about the strange rocks and bugs I've found while walking along it.

The completion of this paper was supported by the National Science Foundation under Grant No. DMS–2424139, while the author was in residence at the Simons Laufer Mathematical Sciences Institute in Berkeley, California, during the Spring 2026 semester. During the project, the author was partially supported by a Simons Dissertation Fellowship.

\section{Background}\label{sec:background}

We now review the background we'll need for the paper. In particular, we go over limit sets and accumulation sets, leaf spaces of foliations, Anosov flows and foliations, universal circles, and quasigeodesic flows. 

\subsection{Topological preliminaries and notation}\label{sec:limit_background}
Given a sequence $\{x_j\}_{j\in\N}$ of points in a topological space $X$, we denote by 
\[
\acset_{j\to\infty}x_j
\]
the set of \defn{accumulation points} of the sequence, i.e. the set of points $y\in X$ such that for all open neighborhoods $U$ of $y$, infinitely many elements of the sequence lie in $U$. We denote by 
\[
\limset_{j\to\infty}x_j
\]
the set of \defn{limit points} of the sequence, i.e. the set of points $y\in X$ such that for all open neighborhoods $U$ of $y$, all but finitely many elements of the sequence lie in $U$. As we will be working in non-Hausdorff spaces, it may happen that 
\(
\limset_{j\to\infty}x_j
\)
contains more than one point. In the case that 
\(
\limset_{j\to\infty}x_j
\) 
is a singleton, we may denote it by 
\[
\lim_{j\to\infty}x_j
\] 
and call it the \defn{limit} of the sequence.

\subsection{Leaf spaces}\label{sec:leaf_background}
Given a two-dimensional foliation $\mc F$ of a 3-manifold $M$, we define the \defn{leaf space} of $\mc F$ by 
\[
\mc L=\wt M/\wt{\mc F},
\]
where $\wt M$ is the universal cover of $M$ and $\wt{\mc F}$ is the lift of $\mc F$ to $\wt M$. In general, $\mc L$ is a simply connected, possibly non-Hausdorff 1-manifold. The foliation $\mc F$ is said to be $\R$\defn{-covered} if $\mc L\cong\R$, and \defn{non $\R$-covered} otherwise, i.e. if $\mc L$ is non-Hausdorff. A \defn{cataclysm} is a maximal set of pairwise nonseparated leaves. Two leaves within a cataclysm are said to be \defn{branching}, and a leaf that is not part of a cataclysm is a \defn{nonbranching leaf}. A pair of leaves $\l,\mu$ are said to be \defn{comparable} if there is an embedded interval connecting them in the leaf space, and \defn{incomparable} otherwise (thus, nonseparated leaves are incomparable).

A coorientation on $\mc F$ gives an orientation on $\mc L$. All our foliations will be coorientable, so we always assume $\mc L$ is oriented. In figures, we draw the leaf space so that the orientation points up. Given a pair of comparable leaves $\l,\mu$, we say $\l$ is \defn{below} $\mu$, written $\l<\mu$, if the orientation on $\mc L$ restricted to the embedded interval connecting $\l$ to $\mu$ points from $\l$ to $\mu$. Similarly, we say $\l$ is \defn{above} $\mu$, written $\l>\mu$, if $\mu$ is below $\l$. When $\l<\mu$, we denote by $[\l,\mu]$ the interval in $\mc L$ connecting $\l$ to $\mu$. Given a pair of nonseparated leaves $\l,\l'$, we say they are \defn{branching from above} if they are limit points of a sequence of leaves $\{\l_j\}_{j\in\N}$ above $\l$ and $\l'$. Similarly, we say they are \defn{branching from below} if they are simultaneously limited onto by leaves below them.

\begin{figure}[h]
	\centering
	\includegraphics[width=0.8\linewidth]{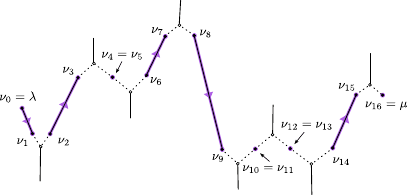}
	\caption{In purple, the zigzag path $\gamma$ connecting $\l$ and $\mu$. The arrows denote the orientation on $\gamma$. The `degenerate' orientations at $\nu_4$, $\nu_{10}$, and $\nu_{12}$ can be deduced from the fact that the orientations alternate between up and down.}\label{fig:broken_path}
\end{figure}

Any pair of points $\l$ and $\mu$ in $\mc L$ can be connected by a unique minimal broken oriented path $\gamma$ from $\l$ to $\mu$. We will call these \defn{zigzag paths}. We use \defn{zigzag ray} to refer to a zigzag path that is infinite in one direction.

In general, $\gamma$ contains a sequence of (possibly nondistinct) leaves $\l=\nu_0,\nu_1,\ldots,\nu_n=\mu$ as in \Cref{fig:broken_path} such that 
\begin{enumerate}
	\item for $j$ odd, $\nu_j$ and $\nu_{j+1}$ are nonseparated, and
	\item for $j$ even, $\nu_j$ and $\nu_{j+1}$ are comparable, and $\gamma$ is the union of the intervals connecting $\nu_j$ and $\nu_{j+1}$.
\end{enumerate}
Intuitively, $\gamma$ travels along an embedded interval, jumps across some number of cataclysms, travels along another interval, and so on. For $j$ even, $\gamma$ determines an orientation on the interval from $\nu_j$ to $\nu_{j+1}$ (even when this is a degernate interval), and the orientations alternate between agreeing and disagreeing with the orientation on $\mc L$. The $\nu_j$'s are called \defn{breakpoints} of the zigzag path $\gamma$. For $j\neq n$ odd, $\nu_j$ is called a \defn{launching leaf}, and $\nu_{j+1}$ is the corresponding \defn{landing leaf}, since $\gamma$ launches from $\nu_j$ and lands on $\nu_{j+1}$ as it jumps across the cataclysm.

Let $\mc H\colon\mc L\to\mc H(\mc L)$ denote the Hausdorffification of $\mc L$, i.e. the minimal quotient that identifies nonseparated points in $\mc L$. We denote by $\Ends(\mc L)$ the \emph{Freudenthal space of ends} of $\mc H(\mc L)$. In our case, we can think of $\Ends(\mc L)$ as equivalence classes of properly embedded rays. See \cite[Section 2.2]{calegari2024zippers} for a discussion of the ends of a topological $\R$-tree that works in our setting.

\subsection{Anosov flows and foliations}\label{sec:anosov_background}

A \defn{topological Anosov flow} on a 3-manifold $M$ is a flow $\varphi$ preserving a pair of transverse two-dimensional foliations $W^s$ and $W^u$, called the \defn{weak stable and unstable foliations}, such that the following three conditions hold.
\begin{enumerate}
	\item Leaves of $W^s$ and $W^u$ are foliated by flow lines and intersect along flow lines.
	
	\item For all sufficiently small $\vep>0$, for $x\in\l$ a leaf of $W^s$ (respectively, a leaf of $W^u$), and for all $y\in \l_\vep(x)$, where $\l_\vep(x)$ is the connected component of $B_\vep(x)\cap\l$ containing $x$, there is an increasing homeomorphism $\sigma\colon\R\to\R$ with $\sigma(0)=0$ such that $d(\phi_t(x),\phi_{\sigma(t)}(y))\to 0$ as $t\to\infty$ (respectively, as $t\to-\infty$).

	\item For all sufficiently small $\vep>0$, there exists $\delta>0$ such that for $x\in\l$ a leaf of $W^s$ (respectively $W^u$), if $y\in\l_\vep(x)$ and $y$ isn't in the same $\vep$-local orbit of $x$, then there exists some time $t$ such that for any $\sigma\colon\R\to\R$ as above, $d(\phi_t(x),\phi_{\sigma(t)}(y))\geq\delta$.
\end{enumerate}

The foliations $W^s$ and $W^u$ are called \defn{Anosov foliations}. The flow $\varphi$ is said to be \defn{$\R$-covered} if $W^s$ (equivalently, $W^u$) is, and \defn{non $\R$-covered} otherwise. We denote the respective leaf spaces by $\mc L^s$ and $\mc L^u$.

Given an Anosov flow $\varphi$, we construct the \defn{flow space} (or \defn{orbit space}) $\orb$ by lifting to the universal cover and collapsing flow lines, i.e.
\[
\orb=\wt M/\wt\varphi
\]
where $\wt M$ is the universal cover of $M$ and $\wt \varphi$ is a lift of $\varphi$ to $\wt M$. We let
\[
\Theta\colon\wt M\to\orb
\]
be the quotient map. By work of Barbot \cite{barbot1995caracterisation} and Fenley  \cite{fenley1994anosov}, the flow space is topologically a plane. The lifted two-dimensional foliations $\wt W^{s/u}$ are preserved by the lifted flow, and therefore descend to one-dimensional foliations $\orb^{s/u}$ of $\orb$, still called the \defn{stable and unstable foliations}. These foliations are typically quite complicated; see \Cref{fig:orb} for an idea of what they might look like. 

\begin{figure}[h]
	\centering
	\includegraphics[width=0.45\linewidth]{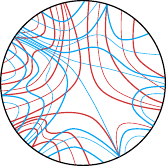}
	\caption{The orbit space of a non $\R$-covered Anosov flow (reproduced from \cite[Figure 3]{buckminster2025universal}). Per \Cref{con:redblue}, $\orb^s$ is in blue and $\orb^u$ is in red.}\label{fig:orb}
\end{figure}

When the foliations $W^{s/u}$ are orientable and coorientable, we get orientations and coorientations on $\orb^{s/u}$. The action $\pi_1(M)\curvearrowright\wt M$ descends to an action $\pi_1(M)\curvearrowright\orb$ preserving the foliations $\orb^{s/u}$. We set the following convention on colors and notation.

\begin{convention}\label{con:redblue}
	In the orbit space, we will always draw leaves of $\orb^s$ in blue and leaves of $\orb^u$ in red. Following \cite{buckminster2025universal}, we use Greek letters (such as $\l$, $\mu$, and $\nu$) for leaves in $\wt W^{s/u}$, and use the corresponding Latin letter (respectively $\ell$, $m$, and $n$) for the corresponding leaf in $\orb^{s/u}$. We preserve subscripts and superscripts. For instance, if $\l^u_j$ is a leaf of $\wt W^u$, then $\Theta(\l^u_j)=\ell^u_j$, a leaf of $\orb^u$.
\end{convention}

By work of Fenley \cite{fenley2012ideal}, the flow space can be compactified with a circle $\partial\orb$, called the \defn{flow space boundary}, such that $\overline{\orb}=\orb\cup\partial\orb$ is homeomorphic to the closed disk. The action of $\pi_1(M)$ on $\orb$ extends continuously to an action on $\overline{\orb}$. 

A leaf $\ell$ of $\orb^s$ or $\orb^u$ will have two `endpoints' in $\partial\orb$. When $\ell$ is oriented, we denote the forward endpoint by $\ell^+$ and the backward endpoint by $\ell^-$. We denote the pair of endpoints by $\partial\ell$. We denote by $\partial\orb^{s/u}$ the set of endpoints of leaves of $\orb^{s/u}$. Given an arbitrary subset $X$ of $\orb$, we denote by $\overline{X}$ the closure in $\overline{\orb}$, and we denote by $\partial X$ the set $\overline{X}\cap\partial\orb$.

\subsubsection{Structures in the orbit space}
A \defn{perfect fit rectangle} is a proper embedding 
\[
[0,1]\times[0,1]\ssm\{(1,1)\}\hookrightarrow\orb
\]
such that the horizontal and vertical foliations are sent to portions of leaves of $\orb^s$ and $\orb^u$, respectively. The leaves $\ell^s$ and $\ell^u$ contained in the image of $\{(1,t)\,\mid \, t\in[0,1)\}$ and $\{(t,1)\,\mid \, t\in[0,1)\}$ are said to \defn{make a perfect fit}. See \Cref{fig:perfectfit}.

\begin{figure}[h]
	\centering
	\includegraphics[width=0.8\linewidth]{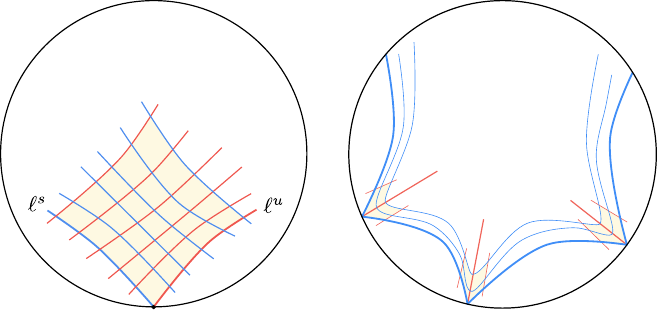}
	\caption{On the left, a perfect fit rectangle. The leaves $\ell^u$ and $\ell^s$ make a perfect fit. On the right, a branching chain, with adjacent perfect fit rectangles in yellow.}\label{fig:perfectfit}
\end{figure}

Nonseparated leaves of $\orb^{u}$ (or $\orb^s$), also called branching leaves, form \defn{branching chains}, which are ordered sets of nonseparated leaves such that each shares an endpoint in $\partial\orb$ with the leaves on either side of it, as in \Cref{fig:perfectfit}. Between adjacent leaves $\ell^u_1$, $\ell^u_2$ in a branching chain of $\orb^u$, there is a stable leaf $\ell^s$ making a perfect fit with each of $\ell^u_1$ and $\ell^u_2$.

We refer the reader to the recent book by Barthelm\'{e}--Mann for more details on Anosov flows \cite{barthelme2025pseudo}.

\subsection{Universal circles}\label{sec:uc_background}
A \emph{universal circle} for a foliation on a 3-manifold $M$ is a way of producing an action of $\pi_1(M)$ on a circle that is related to the geometry of the foliation at infinity. They were first defined by Thurston, and later made rigorous and studied further by Calegari--Dunfield in \cite{CalDun_UC}. We review general universal circles only briefly, as the \emph{stitching map} from \cite{buckminster2025universal} makes the study of universal circles for Anosov foliations much more concrete. We refer the reader to \cite{CalDun_UC} for a thorough treatment of general universal circles.

Let $\mc F$ be a taut, coorientable foliation on a closed, atoroidal 3-manifold $M$, and let $\wt {\mc F}$ be the lift of $\mc F$ to the universal cover $\wt M$. By Candel's theorem on uniformizing foliations \cite{Candel_uniformization}, we can put a metric on $M$ such that the restriction to any leaf of $\mc F$ is hyperbolic. Thus, any leaf $\l$ of $\wt{\mc F}$ can be identified with $\HH^2$, and therefore compactified with an \defn{ideal circle} $\partial_\infty\l$. The set 
\[
\bigcup_{\l\text{ a leaf of }\wt{\mc F}}\partial_\infty\l
\]
can be topologized in a natural way, resulting in the \defn{circle bundle at infinity} $E_\infty$, which is a circle bundle over the leaf space $\mc L$ of $\mc F$. The circles of the form $\partial_\infty\l$ are called \defn{circle fibers}. The action of $\pi_1(M)$ on $\wt M$ by deck transformations defines an action of $\pi_1(M)$ on $E_\infty$ preserving the circle fibers.

\begin{definition}\label{def:uc}
A \defn{universal circle} is a circle $\univ$, together with an action $\pi_1(M)\curvearrowright\univ$ and a collection of \emph{monotone maps} $\{\pi_\l\colon\univ\to\partial_\infty\l\;\mid\;\l\in\mc L\}$ (i.e. maps preserving a circular order) such that the following conditions are satisfied.

\begin{enumerate}
	\item The map $\pi\colon \univ\times\mc L\to E_\infty$ defined by $\pi(z,\l)=\pi_\l(z)$ is continuous.
	\item For all $\gamma\in\pi_1(M)$ and $\l\in\mc L$, the following diagram commutes. 
	\[\begin{tikzcd}
		\univ & \univ \\
		{\partial_\infty\l} & {\partial_\infty\gamma\cdot\l}
		\arrow["{\cdot\gamma}", from=1-1, to=1-2]
		\arrow["{\pi_\l}"', from=1-1, to=2-1]
		\arrow["{\pi_{\gamma\cdot\l}}", from=1-2, to=2-2]
		\arrow["{\cdot\gamma}", from=2-1, to=2-2]
	\end{tikzcd}\]

	\item For incomparable leaves $\l,\mu\in\mc L$, the \emph{core} of the map $\pi_\l$, i.e. the set of points in $\univ$ on which $\pi_\l$ is injective, is contained in a single \emph{gap} of $\pi_\mu$, i.e. the preimage in $\univ$ of a single point in $\partial_\infty\mu$ under $\pi_\mu$.
\end{enumerate}
\end{definition}

By \cite{CalDun_UC}, any such foliation $\mc F$ admits a universal circle called the \emph{leftmost universal circle}, which we now describe.

A \defn{marker} is an embedding $m \colon I\to E_\infty$ that is transverse to the circle fibers, and such that the corresponding leaves of $\wt{\mc F}$ stay a uniformly bounded distance apart along the direction toward infinity that the marker defines. We also use the term `marker' to refer to the image of a marker. A point in $E_\infty$ with a marker through it is called a \defn{marker point}, and is called a \defn{nonmarker point} otherwise.

\begin{figure}[h]
	\centering
	\includegraphics[width=0.6\linewidth]{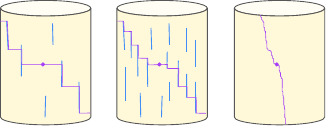}
	\caption{Approximating $s^\ell_p$ by adding in more and more markers and travelling leftmost up, rightmost down without crossing markers (reproduced from \cite[Figure 7]{buckminster2025universal}). Markers are drawn in blue, the point $p$ is the purple dot, and $s^\ell_p$ and its approximants are drawn in purple.}\label{fig:special_section}
\end{figure}

A \defn{section} is a map $s\colon\mc L\to E_\infty$ that intersects each circle fiber exactly once, i.e. a section of $E_\infty$ as a circle bundle over $\mc L$. We often call the image of a section a `section'. An \defn{admissible section} is a section $s$ whose image does not cross markers, though it may coalesce with markers. The \defn{leftmost section} based at a point $p\in E_\infty$, denoted $s^\ell_p$, is the admissible section through $p$ which, over leaves comparable to the leaf containing $p$, travels as left as possible while travelling up from $p$, and travels as right as possible while travelling down from $p$. The \defn{turning corners rule} defines what $s^\ell_p$ does over other leaves in the leaf space. Intuitively, there is a special point (which in our case will be the unique nonmarker point of a leaf) that $s^\ell_p$ lands on as it jumps across a cataclysm. From there, it continues travelling leftmost up and rightmost down from where it landed. Such a section can be defined rigorously by taking limits; see \Cref{fig:special_section} for the idea.

The set of special sections admits an action by $\pi_1(M)$ and a circular order, coming from the fact that two special sections cannot cross. Thus, this circular order can be completed to produce an actual circle $\mc C^\ell$, which is the \defn{leftmost universal circle}. This is a \emph{minimal} universal circle, meaning there is a bijection between points in $\mc C^\ell$ and the sections they define. Thus, we consider sections as points of $\mc C^\ell$. Points in $\mc C^\ell$ are either special sections or \defn{limit sections}, which are limits of special sections under pointwise convergence. The monotone maps are defined by evaluating sections on circle fibers.

Given a set $A\subset E_\infty$ and a set $\gamma\subset\mc L$, we set the notation
\[
A\mid_{\gamma}=A\cap E_\infty\mid_\gamma.
\]
Typically, $A$ will be a section or a marker, and $\gamma$ will be a zizag path or ray.

\subsubsection{Universal circles for Anosov foliations}\label{sec:anosov_uc_background}
We now specialize to the case where $\mc F=W^u$, the unstable foliation of an Anosov flow on $M$. Let $\mc L^u$ be the leaf space of $W^u$, let $E_\infty$ be the circle bundle at infinity for $W^u$, and let $\orb$ be the orbit space of the Anosov flow. For each leaf $\l^u$ of $\wt W^u$, the ideal circle $\partial_\infty\l^u$ contains a unique nonmarker point, which is the common backwards endpoint of every orbit in $\l^u$. These points vary continuously in the leaf space, and trace out a section $\snm$ which we call the \defn{nonmarker section}. Following \cite{buckminster2025universal}, we define the \defn{stitching map} $\Phi\colon\orb\to E_\infty\ssm\snm$ by 
\[
\Phi(\alpha)=\alpha^+,
\]
where $\alpha^+$ is the forwards endpoint of the orbit $\alpha$ in the ideal boundary $\partial_\infty\l^u$ of the leaf $\l^u$ of $\wt W^u$ containing the orbit $\alpha$. Via the following theorem, the stitching map and the structure of the orbit space tell us everything we need to know about $E_\infty$.

\begin{theorem}[Theorem C of \cite{buckminster2025universal}]
	The map $\Phi$ is a $\pi_1(M)$-equivariant homeomorphism sending leaves of $\orb^s$ to markers and sending leaves of $\orb^u$ to sets of the form $\partial_\infty\l\ssm\snm(\l)$ for $\l\in\mc L^u$.
\end{theorem}

In light of this, we set the following convention for drawing $E_\infty$.

\begin{convention}\label{conv:Einfty}
	In figures, we draw the nonmarker section in orange, the circle fibers in red, and the markers in blue, so that the coloring of $\orb$ according to \Cref{con:redblue} lines up with the coloring of $E_\infty$ after applying $\Phi$.
\end{convention}

By convention, whenever we write $\Phi^{-1}(X)$ for $X\subseteq E_\infty$, we really mean $\Phi^{-1}(X\cap\snm)$.

Given a leaf $\l$ of $W^u$ and an interval neighborhood $\tau$ of $\l$ in $\mc L^u$ the circle fiber $\partial_\infty\l$ and $\snm\mid_\tau$ divide a neighborhood of $\snm(\l)$ in $E_\infty\mid_\tau$ into four \defn{quadrants}, which we call the upper left, upper right, lower left, and lower right quadrants, using the orientation from $\orb$. Mapping this picture to $\orb$ using $\Phi^{-1}$, we see that leaves of $\orb^s$ making perfect fits with $\ell$, and branching leaves of $\orb^u$ adjacent to $\ell$ are in one of these four quadrants.

\subsection{Quasigeodesic flows}\label{sec:quasigeodesic_background}

We now review the necessary background on quasigeodesic flows. We restrict to the case where the flow is Anosov for simplicity, but much of this still holds for general quasigeodesic flows.  

Let $M$ be a closed hyperbolic 3-manifold. Since $M$ is hyperbolic, we have $\wt M\cong\HH^3$, and thus $\wt M$ can be compactified with an ideal 2-sphere $S^2_\infty\cong\partial_\infty\HH^3$. By work of Fenley \cite{fenley2022nonrcoveredanosovflows}, an Anosov flow on $M$ is \defn{quasigeodesic}, meaning that (parameterized) orbits of $\wt\varphi$ stay uniformly bounded distance from geodesics in $\wt M$, if and only if the flow is non $\R$-covered. Let $\varphi$ be a non $\R$-covered Anosov flow on $M$.  Since $\phi$ is quasigeodesic, any flowline $\alpha$ of $\wt\varphi$ has well-defined forwards and backwards endpoints $\alpha^\pm\in S^2_\infty$, which are necessarily distinct. Using this, we define the \defn{endpoint maps} $e^\pm\colon\orb\to S^2_\infty$ by 
\[
e^{\pm}(\alpha)=\alpha^\pm.
\]
Note that $e^+$ is constant on leaves of $\orb^s$, while $e^-$ is constant on leaves of $\orb^u$.

Frankel showed that $e^+$ and $e^-$ can be continuously extended to maps 
\[
\epm\colon\overline\orb\to S^2_\infty
\]
and that these maps agree on $\partial\orb$; see \cite{Frankel_thesis}. We let $e\colon\partial\orb\to S^2_\infty$ denote the common map on the boundary. Thus, the maps $\epm$ agree with $e^\pm$ on $\orb$ and with $e$ on $\partial\orb$. The map $e\colon\partial\orb\to S^2_\infty$ defines a $\pi_1(M)$-equivariant sphere-filling curve, generalizing the Cannon--Thurston maps coming from pseudo-Anosov suspension flows.

It will be important for us the understand the preimages of points in $S^2_\infty$ under $\eplus$ and $\emin$. Following Frankel in \cite{frankel_closedorbits}, given a point $z\in S^2_\infty$, we call 
\[
\mc X=\left(\eplus\right)^{-1}(z)\cup\left(\emin \right)^{-1}(z)
\]
the \defn{master set rooted at $z$}. As $\eplus$ (resp. $\emin$) is constant on leaves of $\orb^s$ (resp. $\orb^u$), we have that $\mc X\cap\orb$ is a union of leaves of $\orb^{s/u}$. Since the flow is quasigeodesic, a flowline has distinct forward and backward limit points on $S^2_\infty$. Thus, $\mc X\cap\orb$ is a union of \emph{disjoint} leaves of $\orb^{s/u}$. Since $\eplus$ and $\emin$ agree on $\partial\orb$, we have that if two leaves of $\orb^{s/u}$ share an endpoint in $\partial\orb$, then either both are in $\mc X$ or neither are. 

Putting all this together, we see that if $\mc X$ contains a point $p\in\orb$, then $\mc X$ contains either $\ell^s$, the stable leaf through $p$, or $\ell^u$, the unstable leaf through $p$. Suppose $\mc X$ contains $\ell^s$. Then $\mc X$ also contains any leaves of $\orb^{s/u}$ connected to $\ell^s$ by a chain of leaves sharing endpoints in $\partial\orb$, as in \Cref{fig:master_set}, as well as the endpoints of all these leaves. By \cite[Lemma 5.2]{frankel_closedorbits}, master sets are connected. Thus, this process builds the entire master set. If $\mc X$ does not contain any points in $\orb$, then it must be a single point in $\partial\orb$.

By work of Fenley \cite{Fen16, fenley1998structure}, we have the following.

\begin{lemma}\label{lem:master_set_finite}
	For $\varphi$ a non $\R$-covered Anosov flow on $M$ a closed, hyperbolic 3-manifold, every master set contains finitely many leaves of $\orb^{s/u}$.
\end{lemma}

\begin{proof}
	By \cite[Main Theorem]{Fen16}, no closed orbit of $\phi$ is nontrivially homotopic to itself, and the number of closed orbits in a free homotopy class is  bounded. Suppose $\mc X$ were a master set with infinitely many leaves. Either $\mc X$ contains infinitely many branching leaves, or it does not.
	
	If $\mc X$ contains infinitely many branching leaves, then some element $g\in\pi_1(M)$ sends a branhing leaf $\ell$ in $\mc X$ to a different branching leaf $\ell'$ in $\mc X$, and therefore fixes $\mc X$. This is because the number of branching leaves of $\orb^{s/u}$ is finite up to the $\pi_1(M)$ action by \cite[Theorem E]{fenley1998structure}. There is a second element $f\in\pi_1(M)$ fixing $\mc X$ because branching leaves are periodic by \cite[Theorem A]{fenley1998structure}. Thus, $f$ and $g$ generate a $\Z^2$ subgroup of $\pi_1(M)$, contradicting the fact that $\pi_1(M)$ is Gromov-hyperbolic.
	
	If $\mc X$ contains only finitely many branching leaves, then $\mc X$ contains an infinite chain of perfect fits. By \cite[Theorem C]{Fen16}, this produces infinitely many closed orbits in a free homotopy class, contradicting the fact that $\phi$ is quasigeodesic.
\end{proof}

\begin{figure}[h]
	\centering
	\includegraphics[width=0.4\linewidth]{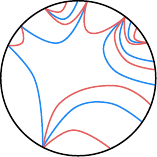}
	\caption{A master set in $\orb$. All leaves of $\orb^{s/u}$ making perfect fits with the leaves shown appear in the figure.}\label{fig:master_set}
\end{figure}

Given a master set $\mc X$ rooted at $z$, the stable leaves in $\mc X$ and their endpoints form the full preimage of $z$ under $\eplus$, and the unstable leaves and their endpoints form the full preimage of $z$ under $\emin$. This is because $\eplus$ and $\emin$ are constant on stable and unstable leaves respectively, and $\eplus$ and $\emin$ agree on $\partial\orb$.

\section{Continuous extensions of leaves}\label{sec:i}

Throughout the paper, fix a non $\R$-covered Anosov flow $\varphi$ on a closed, oriented hyperbolic 3-manifold $M$. Let $\wt \varphi$ be a lift of $\varphi$ to the universal cover $\wt M\cong\HH^3$. Let $S^2_\infty$ denote the ideal 2-sphere boundary of $\wt M$. Let $W^{s/u}$ denote the stable and unstable foliations of $\varphi$ in $M$, and let $\wt W^{s/u}$ be lifts to $\wt M$. Furthermore, assume $W^u$ is oriented and cooriented. Let $\orb$ be the orbit space of $\varphi$.

As in \Cref{sec:quasigeodesic_background}, the flow $\varphi$ is quasigeodesic by \cite{fenley2022nonrcoveredanosovflows}. Given a leaf $\l$ of $\wt W^u$, this allows us to define a continuous extension 
\[
i_\l\colon\partial_\infty\l\to S^2_\infty
\]
of the inclusion $\l\hookrightarrow \wt M$ as follows. For the nonmarker point $\snm(\l)$, we define
\[
i_\l(\snm(\l))=\alpha^-
\]
for $\alpha$ any orbit in $\l$. For a marker point $z\in\partial_\infty\l$, let $\alpha_z$ be orbit in $\l$ limiting to $z$ in forward time. Then we define
\[
i_\l(z)=\alpha_z^+,
\]
which is a continuous extension of the inclusion $\l\hookrightarrow \wt M$.

We can see this as follows. Using our conventions for notation, we have $\ell=\Theta(\l)$, a leaf of $\orb^u$. Given a point $z\in\partial_\infty\l\ssm\{\snm(\l)\}$, again let $\alpha_z$ be the orbit in $\l$ limiting to $z$ in the future. Then $\Theta(\alpha_z)$ is a single point in $\orb$, and 
\[
i_\l(z)=e^+\circ\Theta(\alpha_z).
\]
Moving along the line $\ell$ and applying $e^+$ traces out $i_\l(\partial_\infty\l\ssm\{\snm(\l)\})$. As we approach $\ell^\pm\in\partial\orb$, the points $e^+\circ\Theta(\alpha_z)$ converge to $e(\ell^\pm)$, and we have 
\[
i_\l(\snm(\l))=e(\ell^\pm).
\]
Equivalently, for any point $p\in\ell$, we could take
\[
i_\l(\snm(\l))=e^-(p).
\]

We could further rephrase things to use the stitching map 
\[
\Phi\colon \orb\to E_\infty.
\]
 To do this, we define an extension of $\Phi^{-1}$,
\[
\Psi\colon E_\infty\to\overline{\orb}
\]
by 
\[
\Psi\mid_{E_\infty\ssm\snm}=\Phi^{-1}
\]
and 
\[
\Psi(\snm(\l))=\ell^+.
\]
Of course, the choice of $\ell^+$ over $\ell^-$ was arbitrary, and the map $\Psi$ is not continuous. However, it allows us to succinctly rephrase $i_\l$ as
\[
i_\l=\eplus\circ \Psi\mid_{\partial_\infty\l}.
\]

For convenience, we collate each of these maps $i_\l$ into a single map $i$ as defined below.

\begin{definition}\label{def:i}
	For $\lambda\in\mc L^u$, let 
	\[
	i_\lambda\colon \partial_\infty\lambda\to S^2_\infty
	\]
	be the continuous extension of $\lambda\hookrightarrow \wt M$. Let
	\[
	i\colon E_\infty\to S^2_\infty
	\] 
	be the (discontinuous) function defined by 
	\[
	i(z)=i_\lambda(z)
	\]
	for $z\in\partial_\infty\lambda\subset E_\infty$.
\end{definition}

Using the characterization of $i_\l$ in terms of $\Psi$, we have 
\[
i=\eplus\circ\Psi.
\]

The map $i$ is not continuous, as illustrated by the following example. Let $\{p_j\}_{j\in\N}$ be a sequence of points in $\orb$ limiting to a point $p\in\orb$ on a branching unstable leaf $\ell_1^u$ nonseparated from a leaf $\ell_2^u$, as in \Cref{fig:i_disc}.

\begin{figure}[h!]
	\centering
	\includegraphics[width=0.5\linewidth]{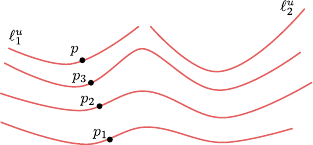}
	\caption{The sequence of points $\{p_j\}_{j\in\N}$ in $\orb$ approaching $p$ in $\ell^u_1$, a branching leaf.}\label{fig:i_disc}
\end{figure}

Then using the stitching map to transfer this setup to $E_\infty$, the sequence $\{\Phi(p_j)\}_{j\in\N}$ approaches both $\Phi(p)$ and $\snm(\l_2^u)$. Under $i$, we have 
\[
i\circ\Phi(p)=e^+(p)
\]
and
\[
i(\snm(\l_2^u))=e^-(\ell_2^u)=e^-(p).
\]
However, $e^+(p)\neq e^-(p)$ since $\varphi$ is quasigeodesic.

Although $i$ is not continuous, the above example is the only thing that goes wrong, in the sense of the following two lemmas.

\begin{lemma}\label{lem:i_restricted_cts}
	The restricted maps
	\[
	i\mid_{s_\text{nm}}\colon s_\text{nm}\to S^2_\infty
	\]
	and
	\[
	i\mid_{E_\infty\ssm s_\text{nm}}\colon E_\infty\ssm s_\text{nm}\to S^2_\infty
	\]
	are continuous.
\end{lemma}

\begin{proof}
	On $\snm$, the map $i$ is conjugate to the map $e^-\colon\mc L^u\to S^2_\infty$, which is continuous since the map $e^-\colon\orb\to S^2_\infty$ is. On $E_\infty\ssm\snm$, the map $i$ is conjugate to the map $e^+\colon\orb\to S^2_\infty$ via the stitching map, which is continuous.
\end{proof}

\begin{lemma}\label{lem:i_almost_cts}
	Let $\{z_j\in\partial_\infty\l_j\}_{j\in\N}$ be a sequence of points in $E_\infty$ limiting to a set of points $\{y_1,\ldots,y_n\}$ in circle fibers $\{\partial_\infty\mu_1,\ldots,\partial_\infty\mu_n\}$. If either 
	\begin{enumerate}
		\item $y_k\neq\snm(\mu_k)$ for all $k$, or
		\item $y_k=\snm(\mu_k)$ for all $k$,
	\end{enumerate}
	then 
	\[
	\lim_{j\to\infty}i(z_j)=i(y_k)
	\]
	for all $k$.
\end{lemma}

\begin{proof}
	First, suppose we are in case (1). Then we must have $n=1$, as $E_\infty\ssm\snm$ is homeomorphic to a plane via \cite[Theorem C]{buckminster2025universal}, and in particular is Hausdorff. Since
	\[
	y_1=\lim_{j\to\infty}z_j,
	\]
	we must have that past some index, each $z_j$ is contained in $E_\infty\ssm\snm$. Thus,
	\[
	\lim_{j\to\infty}i(z_j)=i\left(\lim_{j\to\infty}z_j\right)=i(y_1)
	\]
	by \Cref{lem:i_restricted_cts}.
	
	Now suppose instead that we are in case (2). Split the sequence $\{z_j\}$ into two subsequences, $\{z_j'\}$ and $\{z_j''\}$, such that 
	\[
	z_j'\in \snm
	\] 
	and 
	\[
	z_j''\in E_\infty\ssm\snm
	\]
	for all $j$. If either sequence is finite, we ignore it. For the sequence $\{z_j'\}$, we have that 
	\[
	\lim_{j\to\infty}i(z_j')=i(y_k)
	\]
	for all $k$ by \Cref{lem:i_restricted_cts}. 
	
	For the sequence $\{z_j''\}$, pick some $\mu_r\in\{\mu_1,\cdots,\mu_n\}$, and let $\tau$ be a transversal determining an open interval neighborhood $\tau\subset\mc L^u$ of $\mu_r$. Since 
	\[
	\limset_{j\to\infty} \l_j=\{\mu_1,\ldots,\mu_n\},
	\]
	we must have that past some index, all the $\l_j$ are contained in $\tau\subset\mc L^u$. Let $\mc S^u(\tau)$ denote the unstable saturation of $\tau$ in $\orb$, and let $\overline{\mc S^u(\tau)}$ denote its closure in $\overline{\orb}$. Since 
	\[
	\Phi^{-1}(z_j'')\in\mc S^u(\tau),
	\]
	we must have that 
	\[
	\acset_{j\to\infty}\Phi^{-1}(z_j'')\subset\overline{\mc S^u(\tau)}.
	\]
	Since 
	\[
	\limset_{j\to\infty} \l_j=\{\mu_1,\ldots,\mu_n\},
	\]
	we in fact have that 
	\[
	\acset_{j\to\infty}\Phi^{-1}(z_j'')\subset\overline{\{m_1,\ldots,m_n\}},
	\]
	where the closure is taken in $\overline{\orb}$. Furthermore, the accumulation set must be contained entirely in $\partial\orb$. This is because $y_k=\snm(\mu_k)$ for every $k$ by assumption. Thus,
	\[
	\acset_{j\to\infty}i(z_j'')=\acset_{j\to\infty}e^+\circ\Phi^{-1}(z_j'')=e(m_r^+)=i(y_r),
	\]
	as $e$ is a continuous extension of $e^+$ to $\partial\orb$.
\end{proof}

\subsection{Extension to the end space}

We can extend $i$ to $\Ends(\mc L^u)$ as follows. Let $\varepsilon$ be an end of $\mc L^u$, and let $\gamma$ be a zigzag ray in $\mc L^u$ limiting to $\vep$, and suppose $\gamma$ is oriented toward $\vep$. For each leaf $\l$ in $\gamma$, let $\spn^+(\ell)$ be the component of $\partial\orb\ssm\{\ell^+,\ell^-\}$ on the positive side of $\ell$, where the coorientation on $\ell$ is induced by the orientation on $\gamma$. 

\begin{lemma}\label{lem:i_hat_wd}
	Let $\vep\in\Ends(\mc L^u)$, and let $\gamma$ be a zigzag ray in $\mc L^u$ limiting to $\vep$. Then the set
	\[
	\bigcap_{\l\in\gamma}\spn^+(\ell)
	\]
	is a single point in $\partial\orb$, and is independent of the chosen $\gamma$.
\end{lemma}

\begin{proof}
	Suppose we have $\vep$ and $\gamma$ as in the lemma statement. Let $\lambda,\mu\in\gamma$ such that $\mu$ is closer to $\vep$ than $\l$. Then 
	\[
	\spn^+(m)\subset\spn^+(\ell),
	\]
	which shows that 
	\[
	\bigcap_{\l\in\gamma}\spn^+(\ell)
	\]
	is a nonempty, possibly degenerate interval. We aim to show it is a single point, so suppose toward contradiction it is a nondegenerate interval $I$ in $\partial\orb$. 
	
	Since $\gamma$ is a properly embedded zigzag ray in $\mc L^u$, we have that in $\orb$,
	\[
	\limset_{\l\in\gamma}\ell=\varnothing.
	\]
	Thus in $\overline\orb$, we must have that
	\[
	\limset_{\l\in\gamma}\ell=I.
	\]
	However, using the density of regular periodic points \cite[Lemma 2.30]{BFM} and the corresponding multi sink-source dynamics on $\partial\orb$ \cite[Proposition 3.7]{BFM}, we can find an unstable leaf $\ell^u$ with both endpoints in the interior of $I$, contradicting that $\{\ell\}_{\l\in\gamma}$ has no limit points in $\orb$. Thus, 
	\[
	\bigcap_{\l\in\gamma}\spn^+(\ell)
	\]
	is a single point in $\partial\orb$.
	
	To see that this point doesn't depend on the choice of ray $\gamma$, note that since $\mc L^u$ is simply connected, any two zigzag rays $\gamma$ and $\gamma'$ to $\vep$ eventually coincide.
\end{proof}

Thus, we can make the following definition.

\begin{definition}\label{def:i_hat}
	Define
	\[
	f\colon\Ends(\mc L^u)\to \partial\orb
	\]
	by
	\[
	f(\vep)=\bigcap_{\l\in\gamma}\spn^+(\ell)
	\]
	for $\gamma$ a zigzag ray in $\mc L^u$ limiting to $\vep$. Then we define 
	\[
	\wh i\colon\Ends(\mc L^u)\to S^2_\infty
	\]
	by
	\[
	\wh i(\vep)=e\circ f(\vep).
	\]
\end{definition}

\begin{lemma}\label{lem:i_hat_cts}
	Let $\gamma$ and $e$ as above. Then for leaves $\{\l_j\}_{j\in\N}\subset\gamma$ limiting to $\vep$, we have
	\[
	\lim_{j\to\infty}i(z_j)=\wh i(\vep)
	\] 
	for any points $z_j\in\partial_\infty\l_j$.
\end{lemma}

\begin{proof}

	By construction of the map $f$, we have that 
	\[
	\lim_{j\to\infty}\Psi (z_j)=f(\vep).
	\]
	Thus,
	\[
	\lim_{j\to\infty}i (z_j)=\lim_{j\to\infty}\eplus\circ\Psi (z_j)=e\circ f(\vep)=\wh i(\vep).
	\]
\end{proof}

\section{Basepoints of sections}\label{sec:basepoint}

\subsection{Intuition}

Let $s^\ell_p$ be a special section based at a point $p\in\partial_\infty\l$ for some $\l\in\mc L^u$. Then, it seems natural enough to attempt to partially define $\ct$ by 
\[
\ct(s^\ell_p)=i(p).
\]
That is, we use the map $i\colon E_\infty\to S^2_\infty$ to map the basepoint $p$ of $s^\ell_p$ to a point in $S^2_\infty$, and that is where we map the section $s^\ell_p$. This is ultimately how we will define $\ct$ on special sections. Of course, basepoints of special sections are not uniquely defined, so we will have to show that this map is well-defined. Our principal concern in this section, however, will be extending this definition to sections that are not special. If we let $s=s^\ell_p$ and let $\partial_\infty\l$ be the circle fiber containing $p$, then the candidate definition for $\ct(s)$ above becomes
\[
\ct(s)=i\circ s(\l).
\]
In other words, we pick out a leaf $\l\in \mc L^u$ that contains a basepoint of $s$, we evaluate $s$ at that leaf, and then we map it via $i$ to get a point in $S^2_\infty$. To extend this to a non-special section $s$, we need a way of picking out something that `looks like' it contains a basepoint of $s$. We do this by considering the structure of the \emph{leftmost up} and \emph{rightmost down regions} for $s$ in the leaf space. This leads us to a definition of the \emph{base} of a section, which is some subset of $\mc L^u\cup\Ends(\mc L^u)$.

Before setting up everything we need to define the base of a section, we discuss an example of a special section and two limit sections, as well as what the base of each will be, in order to give some intuition for the rest of the section.

\begin{example}\label{ex:base_s}
	
	\begin{figure}[h!]
		\centering
		\includegraphics[width=\linewidth]{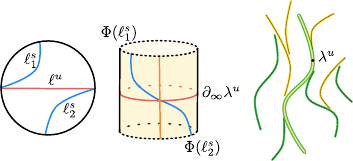}
		\caption{The base of a special section. On the left is the setup in $\orb$, and in the center is a corresponding piece of $E_\infty$. On the right is the coloring of $\mc L^u$, where the base of $s$ is the line colored both yellow and green.}\label{fig:base_exs}
	\end{figure}
	
	Suppose $\ell^u$ is a nonbranching leaf of $\orb^u$ making perfect fits with leaves $\ell^s_1$ and $\ell^s_2$ in the upper left and lower right quadrants, as in \Cref{fig:base_exs}. Suppose further that these three leaves make up an entire master set, i.e. that no other leaves make perfect fits with $\ell^u$, $\ell^s_1$, or $\ell^s_2$. Let $s=s^\ell_{\snm(\l^u)}$, the special section based at $\snm(\l^u)$. Then $s$ follows the marker $\Phi(\ell^s_1)$ as it travels up from $\snm(\l^u)$, and it follows the marker $\Phi(\ell^s_2)$ as it travels down. 
	
	If we let $z$ be a point in the marker $\Phi(\ell^s_1)$, we see that $s$ travels rightmost down from $z$. This is because if we start at $z$ and go rightmost down, we have to stay on the marker until we hit $\snm(\l^u)$. After this point, travelling rightmost down agrees with $s$ by construction. Similarly, $s$ is leftmost up from any point on $\Phi(\ell^s_1)$. Thus, any point in $\Phi(\ell^s_1)\cup\{\snm(\l^u)\}\cup\Phi(\ell^s_2)$ works equally well as a basepoint for $s$. 

	
	Color the leaf space yellow on the places where $s$ is travelling leftmost up, and green in the places where $s$ is travelling rightmost down (we make this more precise in \Cref{def:leftup} and \Cref{def:lu}), as in the right of \Cref{fig:base_exs}. The drawings in the margins are meant to serve as a key. Then there is a line that is colored both yellow and green corresponding to $\Phi(\ell^s_1)\cup\{\snm(\l^u)\}\cup\Phi(\ell^s_2)$, and this is the set which we will define to be the \emph{base} of $s$. Note that for any zigzag path $\gamma$ starting at a point in the base, $\gamma$ will be colored yellow as it travels up and green as it travels down.
\end{example}

\begin{example}\label{ex:base_l1}
	Here is one way to construct a limit section based at an end of $\mc L^u$. Let $\tau$ be properly embedded line in $\mc L^u$ such that $E_\infty\mid_\tau$ contains spiralling markers that travel right up from $\snm$, as in \Cref{fig:base_exl1}. Such pieces of $E_\infty$ arise from \emph{skew-like regions} of the flow space, see \cite[Section 2.11]{barthelme2025pseudo} for more details. The purple markers in the figure are on the frontier of the set of spiralling markers.
	
	Let $\tau^\pm$ denote the ends of the leaf space corresponding to the top and bottom of $\tau$. Let $\{\l_j\}_{j\in\N}$ be a sequence of leaves in  $\tau\subset\mc L^u$ escaping out the end $\tau^-$. Let $s_j=s^\ell_{\snm(\l_j)}$ for all $j$, as in  \Cref{fig:base_exl1}. Since $\mc C^\ell$ is compact, we may pass to a convergent subsequence in $\mc C^\ell$. Let $s$ be the limit of this subsequence; it will consist of all the purple markers on the left half of $\snm$ in \Cref{fig:base_exl1}. Then $s$ won't be a special section, as it is always leftmost up and never rightmost down over the line $\tau$, as in the right of \Cref{fig:base_exl1}. The closest thing to a basepoint is the end $\tau^-\in\Ends(\mc L^u)$. This is because if we let $\gamma$ be any zigzag path limiting to $\tau^-$ on the initial side, we see that $\gamma$ is colored yellow as it travels up and green as it travels down.
	
	Similarly, we may construct a limit section that consists of the purple markers on the right half of $\snm$, and it will be based at the end $\tau^+$.
	
			\AddToShipoutPictureBG*{
		\AtPageLowerLeft{\includegraphics[width=\paperwidth,height=\paperheight]{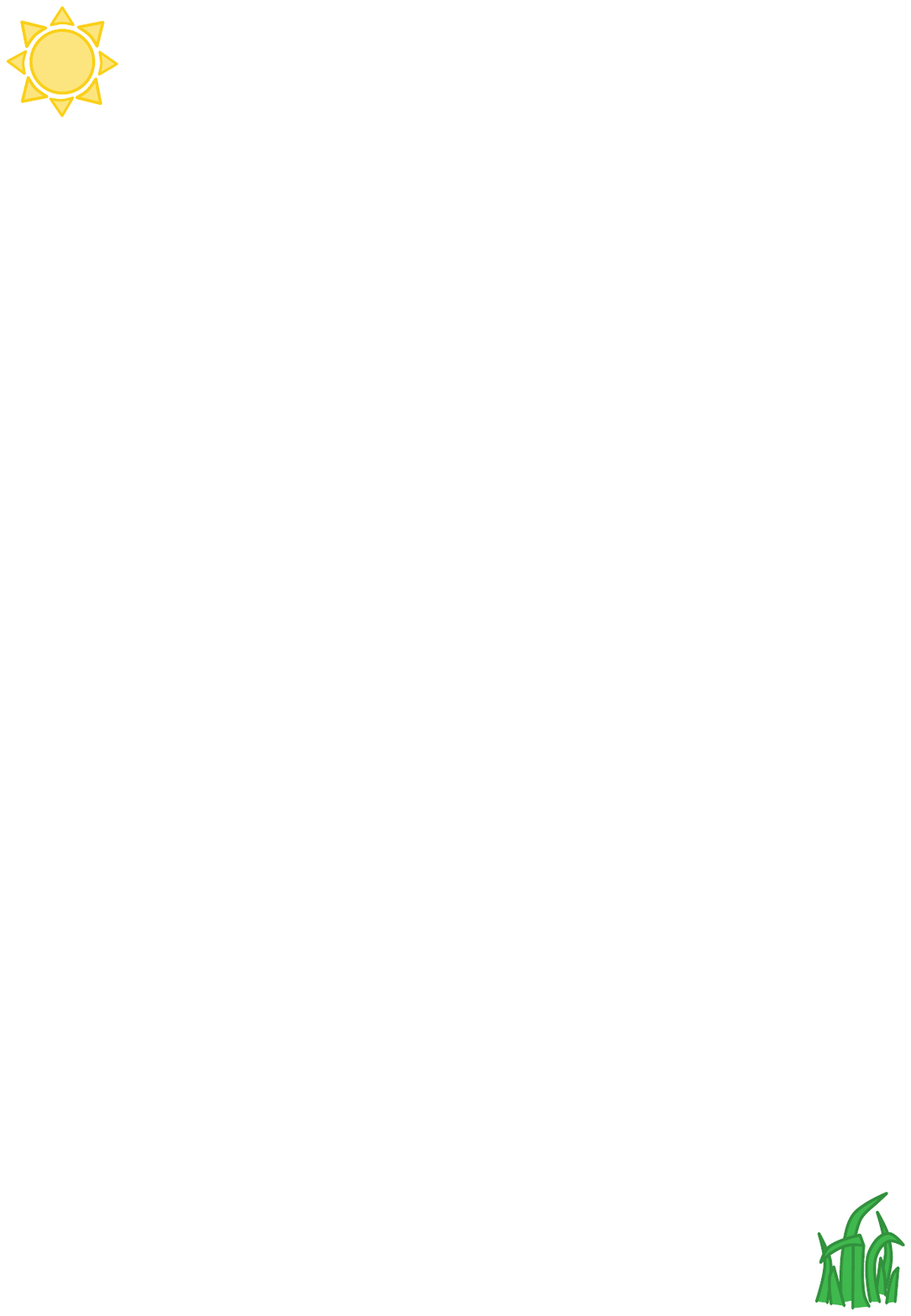}}
	}
	
	\begin{figure}[h!]
		\centering
		\includegraphics[width=0.5\linewidth]{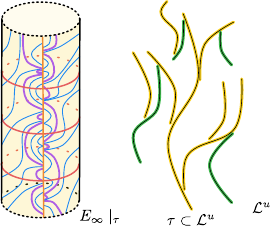}
		\caption{Constructing a limit section `based at' the end $\tau^-$. On the left is $E_\infty\mid_\tau$ with various markers shown. The section $s$ will be all purple markers on the left half of $\snm$. On the right is the coloring induced by $s$ on $\mc L^u$.}\label{fig:base_exl1}
	\end{figure}
\end{example}

\begin{example}\label{ex:base_l2}	
	\begin{figure}[h!]
		\centering
		\includegraphics[width=0.6\linewidth]{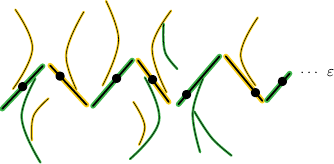}
		\caption{Constructing a limit section `based at' the nonlinear end $\vep$. The zigzag path $\gamma$ is in bold, the points in $\gamma$ show the basepoints for the $s_j$'s, and the coloring is the coloring for $s$.}\label{fig:base_exl2}
	\end{figure}
	
	We can also construct a limit section whose base is a \defn{nonlinear end}, i.e. an end $\vep$ such that any zigzag ray $\gamma$ in $\mc L^u$ limiting to $\vep$ is necessarily broken. Let $\vep$ be such an end, and let $\gamma$ be a zigzag ray limiting to $\vep$. For each maximal embedded interval $I_j$ in $\gamma$, let $s_j$ be a special section based at a point over $I_j$. Thus, the basepoints for the $s_j$'s escape out the end $\vep$. As in the previous example, we may pass to a convergent subsequence limiting to a section $s$. If we observe how $\mc L^u$ is colored using the leftmost up and rightmost down regions of $s$, we see that $s$ is `based at' the end $\vep$, in the sense that zigzag paths starting at $\vep$ are colored correctly. See \Cref{fig:base_exl2}.
\end{example}

			\AddToShipoutPictureBG*{
	\AtPageLowerLeft{\includegraphics[width=\paperwidth,height=\paperheight]{template.pdf}}
}

Although we are yet to formally define the base of a section, we will now state the main result of this section.

\begin{proposition}\label{prop:base}
	Let $s$ be a section of $\mc C^\ell$. Then the base of $s$ is either a point, embedded interval, or line in $\mc L^u$, in which case $s$ is a special section, or it is a single point in $\Ends(\mc L^u)$, in which case $s$ is a limit section.
\end{proposition}

			\AddToShipoutPictureBG*{
	\AtPageLowerLeft{\includegraphics[width=\paperwidth,height=\paperheight]{template.pdf}}
}

\subsection{Leftmost up and rightmost down regions}\label{sec:lurd}

We now formalize the leftmost up and rightmost down regions of $\mc L^u$ for sections of $\mc C^\ell$ and prove some of the basic properties we will need. Throughout, $s$ will be a section of $\mc C^\ell$.

\begin{definition}\label{def:leftup}
	
	We say $s$ is \defn{leftmost up} from a point $p=s(\l)$ if for all leaves $\mu\geq\l$, we have
	\[
	s(\mu)=s^\ell_p(\mu).
	\]
	Similarly, $s$ is \defn{rightmost down} from $p$ if for all $\mu\leq\l$,
	\[
	s(\mu)=s^\ell_p(\mu).
	\]
\end{definition}

\begin{definition}\label{def:lu}

	The \defn{leftmost up region of $s$} is
	\[
	\text{LU(s)}=\{\lambda\in\mc L^u\;\colon\;s\text{ is leftmost up from }\lambda\}.
	\] 
	Similarly, the \defn{rightmost down region of $s$} is
	\[
	\text{RD(s)}=\{\lambda\in\mc L^u\;\colon\;s\text{ is rightmost down from }\lambda\}.
	\] 
\end{definition}

\begin{convention}\label{conv:lurd}
	In figures, we will draw $\lu(s)$ in yellow and $\rd(s)$ in green. The drawings in the margins serve as a key for this convention.
\end{convention}

The following lemma is immediate from the definition.

\begin{lemma}\label{lem:up_sat}

	If $\lambda\in\text{LU}(s)$, then the upwards saturation of $\lambda$, i.e. $\{\mu\in\mc L^u\;\mid\;\mu\geq\l\}$, is in $\text{LU}(s)$. Similarly, for $\mu\in\text{RD}(s)$, the downwards saturation of $\mu$ is in $\text{RD}(s)$.
\end{lemma}

\begin{lemma}\label{lem:lu_or_rd}

	Every leaf $\lambda\in\mc L$ is in at least one of $\lu(s)$ or $\rd(s)$.
\end{lemma}

\begin{proof}
	Suppose toward contradiction that there exists a leaf $\l$ in neither $\lu(s)$ nor $\rd(s)$. Let $\tau$ be a line in $\mc L^u$ containing $\l$, and let $s'=s^\ell_{s(\l)}$. Since $\l\notin \lu(s)$, there is some leaf above $\l$ where $s$ and $s'$ disagree; let $\mu^+$ be the infimum in $\tau$ of such leaves. Similarly, let $\mu^-$ be the supremum in $\tau$ of leaves below $\l$ where $s$ and $s'$ disagree. 
	
	As $s'$ diverges from $s$ at $\mu^+$ travelling up, it heads further left than $s$. Similarly, travelling down, it heads further right than $s$ at $\mu^-$. Thus, $s$ and $s'$ cross over a neighborhood of the interval $[\mu^-,\mu^+]\subset\mc L$, contradicting \cite[Section 6.14]{CalDun_UC}.
\end{proof}

The following lemma will be crucial. Intuitively, we will use it to show that the coloring of $\mc L^u$ alternates as we jump across cataclysms.

\begin{lemma}\label{lem:change_at_cataclysm}

	Let $\lambda$ and $\mu$ be nonseparated leaves. If $\lambda$ and $\mu$ are branching from above, then at most one of $\lambda$ and $\mu$ is in $\lu(s)$. If $\lambda$ and $\mu$ are branching from below, at most one is in $\rd(s)$.
\end{lemma}

\begin{figure}[h!]
	\centering
	\includegraphics[width=0.5\linewidth]{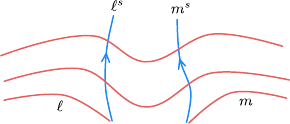}
	\caption{The setup of $\ell$, $m$, $\ell^s$, and $m^s$ in $\orb$ for the proof of \Cref{lem:change_at_cataclysm}.}\label{fig:two_views_setup}
\end{figure}

\begin{figure}[h!]
	\centering
	\includegraphics[width=1\linewidth]{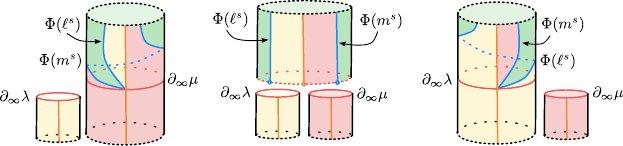}
	\caption{Three views of a neighborhood of $\{\partial_\infty\l,\partial_\infty\mu\}$ in $E_\infty$. The view on the left accurately depicts the topology in a neighborhood of $\partial_\infty\mu$, the right shows the topology around $\partial_\infty\l$, and the center view is `impartial'.}\label{fig:two_views}
\end{figure}

\begin{proof}
	We prove only the first statement; the second is similar. Let $\l$ and $\mu$ be nonseparated leaves in $\mc L$ that are branching from above, and let $\tau\subset\mc L$ be an open interval limiting onto both $\l$ and $\mu$. Suppose $\ell$ and $m$ are arranged as in \Cref{fig:two_views_setup}, i.e. so that $\ell$ is to the left of $m$. Let $\ell^s$ be the leaf of $\orb^s$ making a perfect fit with $\ell$ in $\ell$'s upper right quadrant, and let $m^s$ be the stable leaf making a perfect fit with $m$ in its upper left quadrant. It is possible that $\ell^s=m^s$.
	
	Suppose $s$ is a section of $\mc C^\ell$ that is leftmost up from both $s(\l)$ and $s(\mu)$. If $s(\mu)=\snm(\mu)$, then $s$ must take $\Phi(m^s)$ as it travels up from $\mu$ to be leftmost up from $\mu$. However, following $\Phi(m^s)$ is not leftmost up from $\l$ by \Cref{fig:two_views}.
	
	If $s$ agrees with a marker on $\mu$, then $s(\l)=\snm(\l)$, since $E_\infty\ssm\snm$ is Hausdorff. However, by \Cref{fig:two_views}, no marker through $\partial_\infty\mu$ is leftmost up from $\snm(\l)$. Thus, in either case we derive a contradiction.
\end{proof}

\begin{lemma}\label{lem:marker_sat}
	Suppose $s$ is on a marker $m$ at some leaf $\l\in\mc L^u$. If $\l\in \lu(s)$ (resp. $\l\in\rd(s)$), then every leaf $\nu$ intersecting $m$ is contained in $\lu(s)$ (resp. $\rd(s)$).
\end{lemma}

\begin{proof}

	Suppose $m$ and $\l$ are as in the lemma statement, and suppose $\l\in\lu(s)$.  We aim to show that for all leaves $\nu$ intersecting $m$, $s$ is leftmost up from $\nu$. 
	
	The marker $m$ determines an embedded open interval $m\subseteq\mc L^u$ given by the leaves whose ideal circles intersect $m$ in $E_\infty$. Suppose toward contradiction that $m\nsubseteq\lu(s)$, so there exists some $\nu\in m\subseteq\mc L^u$ such that $\nu\notin\lu(s)$. By \Cref{lem:lu_or_rd}, $\nu\in\rd(s)$. Let $\nu_{\inf{}}$ be the infimum among points in $m\subseteq\mc L^u$ not in $\lu(s)$.

	The section $s$ is travelling leftmost up as it travels up the marker $m$, since any admissible section has to stay on the marker. It is also leftmost up at all leaves above $\nu_{\inf}$ by \Cref{lem:up_sat}. Thus, the only way $\nu\notin\lu(s)$ is if there is some leaf $\nu_1\in[\nu,\nu_{\inf}]\subset\mc L^u$ that is nonseparated from a leaf $\nu_2$ such that $\nu_1$ and $\nu_2$ are branching from below, and $s$ is not leftmost up at some point as it travels up from $\nu_2$. 
	
	Since $s$ is not leftmost up as it travels up from $\nu_2$, we have $\nu_2\in \rd(s)$ by \Cref{lem:lu_or_rd}. However, $\nu_1$ is in $\rd(s)$ by \Cref{lem:up_sat} since $\nu$ is, and this violates \Cref{lem:change_at_cataclysm}, as $\nu_1$ and $\nu_2$ are branching from below.
\end{proof}

\subsection{The base of a section}\label{sec:base}
Our next goal is to define the `base' of a section. This will end up being either the intersection of $\lu(s)$ and $\rd(s)$ in the case that $s$ is a special section, as in \Cref{ex:base_s}, or it will be a single point in $\Ends(\mc L^u)$, as in \Cref{ex:base_l1} and \Cref{ex:base_l2}.

To do this, we will consider zigzag paths whose orientations agree or disagree with the coloring on $\mc L^u$, in the sense of the following definition.

\begin{definition}\label{def:current}
	Let $\gamma$ be a zigzag path in $\mc L^u$, and let $s$ be a section of $\mc C^\ell$. We say $\gamma$ is \defn{oriented with the $s$-current} if it is contained in $\lu(s)$ as it travels up and $\rd(s)$ as it travels down. To be more precise, let $I\subseteq\gamma$ be a maximal embedded interval. Then we require $I\subset\lu(s)$ if the orientation on $I$ is upwards, and $I\subseteq\rd(s)$ if the orientation is downward. (Even on degenerate intervals, there is an `orientation' that makes sense). Similarly, $\gamma$ is \defn{oriented against the $s$-current} if it is contained in $\rd(s)$ as it travels up and $\lu(s)$ as it travels down.
\end{definition}

As in the examples at the beginning of this section, we want to intuitively define the base of a section as the leaves (or ends) where every zigzag path starting at that leaf is colored correctly. Thus, we define the base as follows.

\begin{definition}\label{def:B}
	For $s$ a section of $\mc C^\ell$, we define the \defn{base} of $s$, denoted $B(s)$, as follows. If $\lu(s)\cap\rd(s)\neq\varnothing$, we let 
	\[
	B(s)=\lu(s)\cap\rd(s).
	\] 
	Otherwise, we let $B(s)$ be the set of points $\lambda\in\mc L^u\cup\Ends(\mc L^u)$ such that for every zigzag path $\gamma$ starting at $\lambda$, $\gamma\ssm\{\l\}$ is oriented with the $s$-current. 
\end{definition}

To show that $B(s)$ is nonempty, we essentially have to show that there is a global source for the $s$-current on $\mc L^u$. The idea behind showing this will be to start at an arbitrary leaf, flow against the $s$-current until we get stuck, and then show that where we got stuck is a global source. \Cref{lem:no_sink_local} shows that there is at most one way to flow against the $s$-current from a given leaf, while \Cref{lem:no_sink_global} shows that a local source for the $s$-current is in fact a global source.

\begin{lemma}\label{lem:no_sink_local}
	Let $s$ be a section of $\mc C^\ell$ and $\lambda\in\mc L^u$. If $\lu(s)\cap\rd(s)=\varnothing$, then for any two nondegenerate zigzag paths $\gamma_1,\gamma_2$ starting from $\l$ and oriented against the $s$-current, there exists a nondegenerate zigzag path that is an initial subpath of both $\gamma_1$ and $\gamma_2$.

\end{lemma}

\begin{proof}
	First, suppose $\l$ is a nonbranching leaf. Then any two zigzag paths travelling up from $\l$ agree on an open interval, as do any two zigzag paths travelling down from $\l$. So, we just need to show a zigzag path $\gamma$ going up from $\l$ and a zigzag path $\gamma'$ going down from $\l$ can't both be oriented against the $s$-current. If they were, the first part of $\gamma$ would be in $\rd(s)$, so $\l$ would be in $\rd(s)$ by \Cref{lem:up_sat}. Similarly, the first part of $\gamma'$ would be in $\lu(s)$, so $\l\in\lu(s)\cap\rd(s)$.
	
	\begin{figure}[h!]
		\centering
		\includegraphics[width=0.7\linewidth]{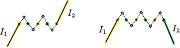}
		\caption{The zigzag path $\alpha$ travels up along $I_1$, jumps across some number of cataclysms, and then travels either up along $I_2$ (as in the left figure) or down along $I_2$ (as in the right figure). In either case, \Cref{lem:change_at_cataclysm} forces the colors of the breakpoints to alternate.}\label{fig:no_sink1}
	\end{figure}

	Now suppose $\l$ is a branching leaf, and let $\gamma_1$ and $\gamma_2$ be two zigzag paths from $\l$ oriented against the $s$-current. Each $\gamma_j$ jumps across some number of cataclysms (possibly zero), before going either up or down on some interval $I_j$. Suppose toward contradiction that $I_1$ and $I_2$ are disjoint. Consider the zigzag path $\alpha$ from a point on $I_1$ to a point on $I_2$. This zigzag path travels along $I_1$, jumps across some number of cataclysms, and then travels along $I_2$, as in \Cref{fig:no_sink1}. Suppose without loss of generality that $\gamma_1$ travels down along $I_1$, which implies that $\alpha$ travels up along $I_1$. Then $I_1\subset \lu(s)$, and by applying \Cref{lem:change_at_cataclysm} to each cataclysm in $\alpha$, we see that the breakpoints in $\alpha$ alternate between being in $\lu(s)$ and being in $\rd(s)$. Thus, we find that $I_2\subset \rd(s)$ if $\alpha$ travels down along $I_2$, and $I_2\subset\lu(s)$ if $\alpha$ travels up along $I_2$. Since the orientation on $I_2$ induced by $\gamma_2$ and $\alpha$ agree, $\gamma_2$ was not actually oriented against the $s$-current.

\end{proof}

\begin{lemma}\label{lem:no_sink_global}
	Let $s$ be a section of $\mc C^\ell$. Let $\gamma$ be a zigzag path in $\mc L^u$ with a nondegenerate initial subpath that is oriented with the $s$-current. Then all of $\gamma$ is oriented with the $s$-current.
\end{lemma}

\begin{proof}
	Suppose otherwise, so there is some zigzag path $\gamma$ with a nondegenerate initial subpath that is oriented with the $s$-current, but such that not all of $\gamma$ is oriented with the $s$-current. Then at some leaf $\l$, $\gamma$ stops being oriented with the $s$-current and starts being oriented against the $s$-current. This contradicts \Cref{lem:no_sink_local}, as travelling from $\l$ along $\gamma$ in either direction gives a zigzag path that is oriented against the $s$-current.
\end{proof}

Putting \Cref{lem:no_sink_local} and \Cref{lem:no_sink_global} together, we can show the following.

\begin{lemma}\label{lem:B_wd}
	For $s$ a section of $\mc C^\ell$, the set $B(s)$ is nonempty.
\end{lemma}

\begin{proof}

	If $\lu(s)\cap\rd(s)\neq\varnothing$, this is clear, so suppose otherwise. The idea is to start at any leaf $\l$ and flow against the $s$-current until we either reach a global source for the $s$-current in $\mc L^u$, or escape out an end of $\mc L^u$. 
	
	So, let $\l\in\mc L^u$, and suppose no zigzag path starting at $\l$ is oriented against the $s$-current. Then we claim $\l\in B(s)$. To see this, let $\gamma$ be a zigzag path from $\l$. By assumption, some initial subpath of $\gamma$ is oriented with the $s$-current. But then by \Cref{lem:no_sink_global}, all of $\gamma$ is oriented with the $s$-current.
	
	So, suppose some zigzag path starting at $\l$ is oriented against the $s$-current. By \Cref{lem:no_sink_local}, initial segments of any two such zigzag paths agree, so we can uniquely develop out a maximal zigzag path or ray $\gamma$ from $\l$ that is oriented against the $s$-current. Either $\gamma$ terminates at some (potentially nonunique) point $\mu\in\mc L^u$, or it escapes out an end $\vep$ of $\mc L^u$.
	
	In the first case, we claim $\mu\in B(s)$. This is because the zigzag path $\gamma$ is assumed to be maximal, so no zigzag path out of $\mu$ is oriented against the $s$-current. Then by \Cref{lem:no_sink_global} again, $\mu\in B(s)$.
	
	In the second case, we claim $\vep\in B(s)$. To see this, let $\alpha$ be a zigzag ray from $\vep$, so $\overline{\alpha}$ is a zigzag ray limiting to $\vep$. Then some terminal portion of $\overline{\alpha}$ agrees with a terminal portion of $\gamma$. Since $\gamma$ is oriented against the $s$-current, this part of $\overline{\alpha}$ is also oriented against the $s$-current. Thus, an initial portion of $\alpha$ is oriented with the $s$-current, so all of $\alpha$ is oriented with the $s$-current by \Cref{lem:no_sink_global}.
\end{proof}

			\AddToShipoutPictureBG*{
		\AtPageLowerLeft{\includegraphics[width=\paperwidth,height=\paperheight]{template.pdf}}
	}

Now that we know the base of a section exists, we turn to proving some properties of $B(s)$, and of sections with a given base. This will allow us to classify sections into three types -- special sections, limit sections based at an end, and limit sections based at a leaf. At the end of this section, we show that sections of the third type do not exist.

\begin{lemma}\label{lem:i_jump}
	
	Let $s$ be a section of $\mc C^\ell$, and let $\gamma$ be a zigzag path oriented with the $s$-current. Then 
	\[
	s(\nu)=\snm(\nu)
	\]
	for all landing leaves $\nu$ in $\gamma$.
\end{lemma}

\begin{proof}

	Let $\nu$ be a landing leaf in $\gamma$, and let $\mu$ be the corresponding launching leaf. Suppose without loss of generality that $\gamma$ goes up from $\nu$, so $\nu$ and $\mu$ are nonseparated leaves that are branching from below. Let $\tau\subset\mc L$ be an open interval limiting from below onto the cataclysm containing $\mu$ and $\nu$, as in \Cref{fig:i_jump}.
	
	\begin{figure}[h!]
		\centering
		\includegraphics[width=0.25\linewidth]{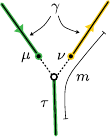}
		\caption{The setup in the proof of \Cref{lem:i_jump}. The arrows indicate the orientation of $\gamma$. The portion of the leaf space that $m$ lives over is indicated.}\label{fig:i_jump}
	\end{figure}

	Since $\gamma$ is oriented with the $s$-current, we have $\mu\in\rd(s)$, $\tau\subset\rd(s)$, and $\nu\in\lu(s)$. Suppose that $s(\nu)$ is on a marker $m$. Then $m$ intersects $E_\infty\mid_\tau$, so the portion of the leaf space that $m$ is over is contained in $\rd(s)$ by \Cref{lem:marker_sat}. However, this includes $\nu$, which contradicts \Cref{lem:change_at_cataclysm}.

\end{proof}

\begin{lemma}\label{lem:type_s}
	For $s$ a section of $\mc C^\ell$, $s$ is special if and only if $\lu(s)\cap\rd(s)\neq\varnothing$. For $s$ special and $\lambda\in B(s)$, we have  $s=s^\ell_{s(\lambda)}$.
\end{lemma}

\begin{proof}
	One direction is straightforward; for $s$ a special section based on $\partial_\infty\l$, we have $\l\in\lu(s)\cap\rd(s)$ by construction. 
	
	Now suppose $s$ is a section of $\mc C^\ell$ with $\l\in\lu(s)\cap\rd(s)$. Then for all leaves $\mu$ comparable to $\l$ we have $s(\mu)=s^\ell_{s(\l)}$. For $\mu$ an arbitrary leaf of $\mc L^u$, let $\gamma$ be the zigzag path from $\l$ to $\mu$. By \Cref{lem:change_at_cataclysm}, we have that the sets $\lu(s)$ and $\lu(s^\ell_{s(\l)})$ agree along $\gamma$, as do $\rd(s)$ and $\rd(s^\ell_{s(\l)})$. Furthermore, each time $\gamma$ jumps across a cataclysm, we must have that $s$ is on the nonmarker section at the landing point of each jump by \Cref{lem:i_jump}.
\end{proof}

\begin{lemma}\label{lem:type_ie}
	Let $s$ be a limit section of $\mc C^\ell$. Then $B(s)$ is either contained in the preimage of a single point under $\mc H\colon\mc L^u\to\mc H(\mc L^u)$ or is a single point in $\Ends(\mc L^u)$.
\end{lemma}
	\AddToShipoutPictureBG*{
	\AtPageLowerLeft{\includegraphics[width=\paperwidth,height=\paperheight]{template.pdf}}
}

\begin{proof}
Suppose otherwise. Then $B(s)$ contains two points $x,y\in\mc L^u\cup\Ends(\mc L^u)$ such that the zigzag path $\gamma$ from $x$ to $y$ contains a nondegenerate embedded interval $I$ (i.e. does not just jump across several cataclysms). Suppose without loss of generality that $\gamma$ travels up along $I$. Then since $x\in B(s)$ and $\gamma$ is a zigzag path from $x$, we should have $I\subset\lu(s)$. However, as $\overline{\gamma}$ is a zigzag path from $y$ to $x$ and $y\in B(s)$, we should have $I\subset\rd(s)$. This contradicts the fact that $s$ is a limit section by \Cref{lem:type_s}.
\end{proof}

\begin{definition}\label{def:types}
	Let $s$ be a section of $\mc C^\ell$. If $s$ is special, we say $s$ is \defn{type (S)} (for `special'). If $s$ is a limit section and $B(s)\subset\mc L^u$, we say $s$ is \defn{type (I)} (for `interior'). If $s$ is a limit section and $B(s)\in\Ends(\mc L^u)$, we say $s$ is \defn{type (E)} (for `end').
\end{definition}

\begin{remark}\label{rem:type_i}

We will shortly prove that in our setting, type (I) sections do not exist. First, we need to develop a bit more structure.
\end{remark}

\begin{lemma}\label{lem:type_s_in_line}
	If $s$ is type (S), then $B(s)$ is connected and contained in an embedded line.
\end{lemma}

\begin{proof}
	Suppose $s$ is type (S), and let $\l\in B(s)$. Suppose $\mu$ is a leaf incomparable with $\l$, and let $\gamma$ be the zigzag path from $\l$ to $\mu$. By applying \Cref{lem:change_at_cataclysm} as we travel across $\gamma$, we find that $\mu\in\lu(s)$ if $\gamma$ was travelling down before the last cataclysm, and $\mu\in\rd(s)$ in the other case. If $\mu$ were in $B(s)$, we would get a contradiction with \Cref{lem:change_at_cataclysm} as we travelled the other way along $\gamma$ from $\mu$ to $\l$. Thus, any two leaves in $B(s)$ are comparable, so $B(s)$ is contained in an embedded line.
	
	The fact that $B(s)$ is connected then follows directly from \Cref{lem:up_sat}.
\end{proof}

\begin{lemma}\label{lem:type_i_in_snm}
	If $s$ is type (I), then $s(\l)=\snm(\l)$ for all $\l\in B(s)$.
\end{lemma}

\begin{proof}

	Suppose $s$ evaluated to a marker point on $\l$. Then the portion of the marker above $s(\l)$ would be in $\lu(s)$, so the entire marker would be in $\lu(s)$ by \Cref{lem:marker_sat}. Similarly, the entire marker would be in $\rd(s)$, so $s$ would be type (S), not type (I).
\end{proof}

We are now ready to show that only type (S) and type (E) sections exist.

\begin{lemma}\label{lem:no_i}
	Every section of $\mc C^\ell$ is type (S) or type (E).
\end{lemma}

\begin{proof}
	Suppose toward contradiction that $s$ is a type (I) section. Let $\l\in B(s)$, and suppose without loss of generality that $\l\notin\lu(s)$. Since $s$ is a type (I) section, we know that for any $\l_\vep>\l$, we have $\l_\vep\in\lu(s)$. Since $\l\notin\lu(s)$, there is a marker $m$ whose lower endpoint is at $s(\l)=\snm(\l)$ and which travels left up, which $s$ does not take. Furthermore, $s$ takes no marker that begins at $\snm(\l)$ and travels up. If it did, the region $\rd(s)$ would extend up from $\l$, and we would not have $\l\in B(s)$. Thus, the setup is as pictured in the left of \Cref{fig:ufto_i}.
	
	\begin{figure}[h!]
		\centering
		\includegraphics[width=0.8\linewidth]{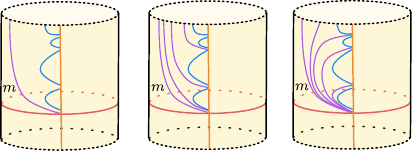}
		\caption{On the left, the setup for the proof of \Cref{lem:no_i}. All the markers drawn in blue are in $s$, while the purple markers are not. The red leaf is $\partial_\infty\l$. In the center, if $\l$ does not have pinching in the upper left quadrant, then $s$ is not leftmost up above $\l$. On the right, the actual setup.}\label{fig:ufto_i}
	\end{figure}

	We must have that $\l$ is branching from above on the left, i.e. has pinching in the upper left quadrant. This is because markers approaching $m$ from the right have lower endpoints approaching $\snm(\l)$. If these lower endpoints aren't eventually equal to $\snm(\l)$, the blue section as drawn isn't leftmost up above $\l$ -- see the center and right side of \Cref{fig:ufto_i}.
	
	Let $\mu_1$ be the leaf that is furthest left in the cataclysm of leaves branching from above containing $\l$. Then we claim $\mu_1\in\lu(s)$; this follows from \Cref{fig:two_views_setup}. 
	
	Now, we may have $\mu_1\notin\rd(s)$. If this is the case, we repeat the above argument to deduce $\mu_1$ has pinching in the lower right quadrant, and switch to the leaf $\mu_2$ that is rightmost in the cataclysm of leaves branching from below containing $\mu_1$. We continue on in this fashion, until we stop at some leaf $\mu_n$. Note that this process must terminate after finitely many steps, as each $\mu_i$ is in the same master set as $\l$, and master sets for Anosov flows in hyperbolic 3-manifolds are finite.
	
	We claim that $\mu_n\in\lu(s)\cap\rd(s)$, and thus $s$ is in fact a type (S) section. To see this, note that we already have one of $\mu_n\in\lu(s)$ or $\mu_n\in\rd(s)$, depending on the parity of $n$. So, suppose $\mu_n\notin\rd(s)$. Then the argument above shows that $\mu_n$ has branching from below and on the right. However, $\mu_n$ was obtained by going `left up, right down' within its cataclysm as much as possible.
\end{proof}

Putting everything together proves \Cref{prop:base}. This implies that a section is a limit section if and only if it is type (E). Thus, we use the terms `limit section' and `special section' in place of 'type (E)' and 'type (S)' from here on out.

\section{Extremal behavior and master sets}\label{sec:qle}
In this section, we first define the Cannon--Thurston map for $\mc C^\ell$ using the base of a section. We then define and discuss \emph{quadrant-local extremality}, which is the idea we will need to prove that this map is well-defined and continuous.

\subsection{Defining the Cannon--Thurston map}\label{sec:ct_def}
We are now ready to define the Cannon--Thurston map for $\mc C^\ell$. See \ref{def:i} and \ref{def:i_hat} for definitions of $i$ and $\wh i$.

\begin{definition}\label{def:ct}
	Define
	\[
	\ct\colon\mc C^\ell\to S^2_\infty
	\]
	as follows. 

For $s$ a section of $\mc C^\ell$, let
\[
\ct(s)=i_\l\circ s(\l)
\]
for $\l$ any leaf of $\mc L^u$ in $B(s)$, or
\[
\ct(s)=\wh i(\varepsilon)
\]
for $\varepsilon$ a point of $\Ends(\mc L^u)$ in $B(s)$.
\end{definition}

This function is not a priori well-defined, as special sections do not always have base a single leaf. And of course, it is not at all clear that this function is continuous. The tool we will need to prove both well-definedness and continuity is the idea of \emph{quadrant-local extremality}, which is a type of extremal behavior that sections can exhibit over zigzag paths.

To give some intuition for this, first consider what is special about the behavior of a special section over its base -- it is both leftmost up and rightmost down in this region. Quadrant-local extremality is a generalization of this property, where a section is only required to be leftmost/rightmost up/down within a quadrant.

The importance of quadrant-local extremality for our purposes is that it relates when portions of sections map to master sets under $\Psi$, as in \Cref{lem:qle_sprig}, with the leftmost up and rightmost down behavior of sections under limits, as in \Cref{prop:ql_lim}.

\subsection{Extremal behavior over zigzag paths}

We now formally define quadrant-local extremality and prove the properties that we will need for \Cref{sec:thm_proof}.

\begin{definition}\label{def:qle}
	Let $\gamma\subset\mc L^u$ be a zigzag path, let $m$ be a marker, and let $m^+$ and $m^-$ be the upper and lower limit points of $m\mid_{\gamma}$ in $\snm$, if they exist. We say $m$ is \defn{quadrant-locally extremal} over $\gamma$ (abbreviated \defn{ql-extremal}) if $m$ is left or rightmost up from $m^-$ among markers with the same endpoint in the same quadrant, and left or rightmost down from $m^+$ among markers with the same endpoint in the same quadrant. See \Cref{fig:qle}.
	
	A section $s$ is \defn{quadrant-locally extremal} over $\gamma$ if the following three conditions hold:
	\begin{enumerate}
		\item Every marker in $s\mid_\gamma$ is ql-extremal over $\gamma$, 
		\item The set $s\cap\snm$ is discrete, and
		\item For all breakpoints $\l\in\gamma$,
		\[
		s(\l)=\snm(\l).
		\]
	
	\end{enumerate}
\end{definition}

\begin{figure}[h!]
	\centering
	\includegraphics[width=0.55\linewidth]{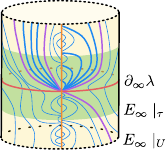}
	\caption{Among all markers ending at $\snm(\l)$, only the purple ones are ql-extremal over $\tau$, where $E_\infty\mid_\tau$ is shaded in green. Markers in bold have $\snm(\l)$ as an endpoint. In the upper right quadrant, the upper endpoints of markers ending at $\snm(\l)$ limit onto $\snm(\l)$.}\label{fig:qle}
\end{figure}

Note that a section is always ql-extremal over its base. In the case of a limit section, this is a vacuous statement, so suppose $s$ is a special section. Then this follows from the fact that $s$ is leftmost up and rightmost down over its entire base.

We now show how ql-extremality relates to master sets.

\begin{lemma}\label{lem:qle_sprig}
	Suppose a section $s$ is ql-extremal over a zigzag path $\gamma$. Then $\Psi(s\mid_\gamma)$ is contained in a master set.
\end{lemma}

\begin{proof}
	Let $I\subset\gamma$ be a maximal embedded interval. We will show $\Psi(s\mid_I)$ is contained in a master set, and then show that this master set is constant as $I$ varies.
	
	If $s\mid_I$ is contained in a single marker, then $\Psi(s\mid_I)$ is contained in a single stable leaf of $\orb^s$, and we're done. So, suppose $s(\l)=\snm(\l)$ for some leaf $\l\in I$. Then by property (2), $s(\l)$ is the endpoint of a marker on any side of $\l$ contained in $I$. Let $m$ be such a marker, and suppose $s(\l)$ is its lower endpoint. By property (1), $m$ is ql-extremal over $I$. Then for any leaf $\nu\in I$ just above $\l$, we see that $m$ intersects $\nu$ but not $\l$. Mapping this to $\overline{\orb}$ via $\Psi$ shows that $\Psi(m)$ makes a perfect fit with $\ell$ or a leaf nonseparated from $\ell$. Thus, $\Psi(s\mid_I)$ is contained in a master set.
	
	Now consider a pair of break points $\nu_1,\nu_2\in\gamma$. By property (3), we have $s(\nu_i)=\snm(\nu_i)$ for $i=1,2$. Since $\nu_1$ and $\nu_2$ are nonseparated, we have that $\Psi(s(\nu_1))$ and $\Psi(s\nu_2))$ are endpoints of nonseparated leaves. Thus, for the adjacent intervals $I_1$ and $I_2$ containing $\nu_1$ and $\nu_2$ respectively, we have that $\Psi(s\mid_{I_1})$ and $\Psi(s\mid_{I_2})$ are contained in the same master set.
\end{proof}

Since $\eplus$ is constant on leaves of $\orb^s$ within a single master set, we have the following corollary.

\begin{corollary}\label{cor:ql_i}
	Suppose a section $s$ is ql-extremal over a zigzag path $\gamma$. Then for any leaves $\l,\mu\in\gamma$, we have
	\[
	i\circ s(\l)=i\circ s(\mu).
	\]
\end{corollary}

Our primary tool for proving ql-extremality is the following proposition.

\begin{proposition}\label{prop:ql_lim}
	Suppose 
	\[
	\lim_{j\to\infty}s_j=s
	\]
	for sections $\{s_j\}_{j\in\N}$ and $s$ of $\mc C^\ell$. Let $\gamma$ be the zigzag path between leaves $\l,\mu\in\mc L^u$. Suppose $\gamma$ is oriented with the $s$-current, but oriented against the $s_j$-current for all $j$. Then $s$ is ql-extremal over $\gamma$.

\end{proposition}

\begin{proof}
	We break the proof into three steps. We first show property (2), i.e. that $s\cap\snm$ over $\gamma$ is discrete. To do this, we need to rule out the situation where $\l\in\gamma$ is a leaf such that $s(\l)=\snm(\l)$, and $s(\l)$ is a limit of nonmarker points in $s$ over $\gamma$, say from above. Pick a small neighborhood $U\subset\gamma$ containing $\l$ such that there is some marker $m$ intersecting every leaf in $U$. Then $E_\infty\mid_U$ is composed of two rectangles $R_\ell$, and $R_r$, that both have arcs of $\snm$ and $m'$ as a pair of opposite sides. Suppose without loss of generality that $U\subset\lu(s)$ and $U\subset\rd(s_j)$ for all $j$. The setup is depicted in \Cref{fig:qle_lim_prop2}.
	
		\begin{figure}[h!]
		\centering
		\includegraphics[width=0.4\linewidth]{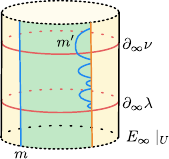}
		\caption{The region $E_\infty\mid U$. The nonmarker section and the marker $m$ divide this into two rectangles, $R_\ell$ (in green) and $R_r$ (in yellow). Other than $m$, all markers drawn are in $s$.}\label{fig:qle_lim_prop2}
	\end{figure}
	
	Let $m'$ be a marker in $s\mid_U$, and let $\nu\in U$ be such that $m'$ intersects $\partial_\infty\nu$. 
	Since 
	\[
	\lim_{j\to\infty}s_j=s,
	\]
	we must have 
	\[
	\lim_{j\to\infty}s_j(\nu)=s(\nu).
	\]
	Eventually, the $s_j$'s must not agree with $s$ at $\nu$. This is because once $m'$ hits $\snm$, the $s_j$'s will continue rightmost down. Therefore, they be constant and not equal to $s$ over $[\l,\nu]\subset U$, as they will be contained in $R_r$, while $s$ is contained in $R_\ell$ over this interval. So, we can pass to a subsequence such that the $s_j$'s approach $s(\nu)$ along $\partial_\infty\nu$ either from the left or the right.
	
	If they approach from the right, we again see that each $s_j$ must cross $\snm$ as it travels down. This is because it is `fenced in' by the marker $m'$. Then between the lower endpoint of $m'$ and $\l$ we can't have that the $s_j$'s are limiting onto $s$, again because they must stay in $R_r$, away from $s$.

	Now suppose the $s_j$'s approach $s(\nu)$ along $\partial_\infty\nu$ from the left. By the above reasoning, they can't cross $\snm$, as they would get stuck in $R_r$. Thus, over $[\l,\nu]$ we must have that each $s_j$ is a single marker $m_j$. Mapping to $\orb$ via $\Phi^{-1}$, we see that the stable leaves $\Phi^{-1}(m_j)$ limit onto the image of $s\mid_U$ under $\Phi^{-1}$, which is an infinite set of stable leaves. These leaves are all nonseparated from each other, as they are in the limit set of a single sequence of stable leaves. This contradicts the fact that on a hyperbolic manifold, sets of nonseparated leaves of $\orb^s$ are finite. Thus, we conclude that property (2) holds.
	
	We now prove property (1), that every marker in $s\mid_\gamma$ is ql-extremal over $\gamma$. Let $\l\in\gamma$ with $s(\l)=\snm(\l)$, and suppose a neighborhood of $\l$ is contained in $\gamma$ (if $\l$ is an endpoint or breakpoint of $\gamma$, the argument below still works, we just ignore leaves either above or below $\l$). Then by property (2), which we have established, we know $s(\l)$ is the upper and lower endpoint of markers $m_1$ and $m_2$, respectively, in $s$. We need to show that $m_1$ (resp. $m_2$) is leftmost or rightmost down (resp. up) among markers with the same endpoint in the same quadrant ending at $\snm(\l)$.
	
	Assume without loss of generality that $\l\in\rd(s)$. Then we know $m_1$ is rightmost down among markers with the same upper endpoint, so we only need to worry about $m_2$. If $m_2$ were not leftmost or rightmost up from $\snm(\l)$ among markers in its quadrant, it would be in the interior of a pinching region. However, this contradicts that $\l\in\lu(s_j)$ for all $j$, since the other markers in the interior of the pinching region are not leftmost up, and therefore are not contained in any $s_j$. This proves property (1).
	
	Finally, we address property (3), that $s(\l)=\snm(\l)$ for all breakpoints $\l\in\gamma$. Let $\l$ be such a breakpoint. If $\l$ is a landing leaf for $\gamma$, then $s(\l)=\snm(\l)$ by \Cref{lem:i_jump}. If $\l$ is a launching leaf for $\gamma$, then it is a landing leaf for $\overline{\gamma}$. Thus, $s_j(\l)=\snm(\l)$ for all $j$ by \Cref{lem:i_jump}, so $s(\l)=\snm(\l)$.
\end{proof}

\section{Proving the main theorem}\label{sec:thm_proof}
In this section, we prove the main theorem of the paper, as well as \Cref{cor:dyn}.

\begin{theorem}\label{thm:ct}
	The map
	\[
	\ct\colon \mc C^\ell\to S^2_\infty
	\]
	defines a continuous, surjective, $\pi_1(M)$-equivariant map. Furthermore, it is a Cannon--Thurston map for $W^u$ in the sense of \Cref{def:ct_for_fol}.
\end{theorem}

The idea behind the proof is this. Well-definedness is essentially immediate from the fact that sections are ql-extremal over their base, so any ambiguity in the choice of basepoint disappears after mapping to $S^2_\infty$. Thus, the main thing to show is continuity. To do this, we suppose we have a sequence of sections $s_j$ converging to a section $s$. If basepoints $\l_j$ of the $s_j$'s converge to a basepoint $\l$ of $s$, we are good to go. However, this is not guaranteed; in fact, plenty of sequences of basepoints of the $s_j$'s will have no convergent subsequence, as $\mc L^u$ is not locally compact. In this case, we can still identify a leaf (or end) $\l_\infty$ that `looks like' a limit point of the basepoints of the $s_j$'s, using the compactness of $\overline{\orb}$. In particular, we will have  
\[
\lim_{j\to\infty}i\circ s_j(\l_j)=i\circ s(\l_\infty).
\]
Then, we show that the section $s$ is ql-extremal over the zigzag path $\gamma$ connecting $\l_\infty$ with some $\l\in B(s)$. By \Cref{cor:ql_i}, this shows that 
\[
i\circ s(\l_\infty)=i\circ s(\l),
\]
which proves continuity. Finally, $\pi_1(M)$-equivariance will follow from the fact that the maps $B$ and $i$ are $\pi_1(M)$-equivariant.

\subsection{Well-definedness}

Proving the map $\ct$ is well-defined is now easy, knowing what we know.

\begin{proposition}\label{prop:ct_wd}
	The map $\ct\colon\mc C^\ell\to S^2_\infty$ is a well-defined function.
\end{proposition}

\begin{proof}
	If $s$ is a special section, then as discussed above, the section $s$ is ql-extremal over its base $B(s)$. Thus, by \Cref{cor:ql_i}, we have 
	\[
	i\circ s(\l)=i\circ s(\mu)
	\]
	for any $\l,\mu\in B(s)$.

	If $s$ is a limit section, then well-definedness follows by \Cref{lem:type_ie}. 
\end{proof}

\subsection{Continuity}
 Suppose $\{s_j\}_{j\in\N}$ is a sequence of sections of $\mc C^\ell$ limiting to $s$. Our aim is to show
\[
\lim_{j\to\infty}\ct(s_j)=\ct(s).
\]
We will do this by showing that every convergent subsequence of $\{\ct(s_j)\}_{j\in\N}$ converges to $\ct(s)$, which suffices since $S^2_\infty$ is compact. So, pass to a convergent subsequence, and relabel so that this subsequence is $\{\ct(s_j)\}_{j\in\N}$. For each $j$, pick some $\l_j\in B(s_j)$, allowing for the possibility that $\l_j\in\Ends(\mc L^u)$.

Consider the sequence 
\[
\{\Psi\circ s_j(\l_j)\}_{j\in\N}\subset \overline{\orb}.
\]
As $\overline{\orb}$ is compact, we can pass to a subsequence $\{s'_j\}_{j\in\N}$ such that 
\[
\{\Psi\circ s'_j(\l'_j)\}_{j\in\N}
\]
converges. Let $z_j'=\Psi\circ s_j'(\l_j')$, and let $z_\infty$ be the limit point. To get a point in $\mc L^u\cup\Ends(\mc L^u)$ from $z_\infty$, we need the following observation.

\begin{lemma}\label{lem:partial_vep}
	If $z\in\partial\orb\ssm\partial\orb^u$, then $z=f(\vep)$ for exactly one $\vep\in\Ends(\mc L^u)$.
\end{lemma}

\begin{proof}
	Let $z\in\partial\orb\ssm\partial\orb^u$, and pick some $z'\in\partial\orb\ssm\partial\orb^u$ with $z'\neq z$. Consider
	\[
	A=\{\ell\in\orb^u\;\mid\;z\text{ and }z'\text{ are in different components of }\partial\orb\ssm\partial\ell\}.
	\]
	Then in $\mc L^u$, $A$ describes a properly embedded minimal broken line limiting to ends $\vep,\vep'\in \Ends(\mc L^u)$ with $f(\vep)=z$ and $f(\vep')=z'$. Thus, the map $f$ is surjective on $\partial\orb\ssm\partial\orb^u$.
	
	To see that $f$ is injective on $\partial\orb\ssm\partial\orb^u$, let $\vep\neq\vep'\in\Ends(\mc L^u)$, let $\gamma$ be the zigzag path between them, and let $\l\in\gamma$. Then $\partial\ell$ separates $\partial\orb$ into two components, with $f(\vep)$ and $f(\vep')$ lying on distinct components. 
\end{proof}

Thus, we can define a function
\[
\wh\Psi\colon\overline{\orb}\to\mc L^u\cup\Ends(\mc L^u)
\]
by sending points in $\orb$ to the unstable leaf they are on, sending points in $\partial\orb^u$ to a choice of unstable leaf with that endpoint, and sending points in $\partial\orb\ssm\partial\orb^u$ to the corresponding end coming from \Cref{lem:partial_vep}. With this in place, we let
\[
\l_\infty=\wh\Psi(z_\infty).
\] 

\begin{lemma}\label{lem:inf_lim}
	We have 
	\[
	\lim_{j\to\infty}i\circ s'_j(\l'_j)=i\circ s(\l_\infty).
	\]
\end{lemma}

\begin{proof}
	By definition,
	\[
	i\circ s(\l_\infty)=\eplus\circ \Psi\circ s(\l_\infty)
	\]
	and
	\[
	i\circ s_j'(\l_j')=\eplus\circ \Psi\circ s_j'(\l_j')
	\]
	for all $j$. Furthermore, we have 
	\[
	\eplus\circ \Psi\circ s(\l_\infty)=\eplus(z_\infty),
	\]
	as either $\Psi\circ s(\l_\infty)=z_\infty$, or $\Psi\circ s(\l_\infty)$ and $z_\infty$ are endpoints of unstable leaves in the same master set. Since 
	\[
	\lim_{j\to\infty}\Psi\circ s'_j(\l'_j)=z_\infty
	\]
	and $\eplus$ is continuous on $\overline{\orb}$, we have
	\[
	\lim_{j\to\infty}\eplus\circ\Psi\circ s'_j(\l'_j)=\eplus(z_\infty),
	\]
	and we are done.
\end{proof}

Pick some $\l\in B(s)$, allowing for the possibility that $\l\in\Ends(\mc L^u)$, and let $\gamma$ be the zigzag path from $\l$ to $\l_\infty$. Our goal now is to show that $s$ is ql-extremal over $\gamma$, which we do via \Cref{prop:ql_lim}. Using \Cref{cor:ql_i}, this will allow us to equate $i\circ s(\l)$, which is how $s$ ``thinks" the Cannon--Thurston map should be defined at $s$, with $i\circ s(\l_\infty)$, which is how the sequence $\{s'_j\}_{j\in\N}$ ``thinks" $\ct$ should be defined at $s$. To do this, we need the following lemma relating the limit point $z_\infty$ we found in $\overline{\orb}$ with what we can see purely in $\mc L^u\cup\Ends(\mc L^u)$.

\begin{lemma}\label{lem:linf_ray}
	Let $\vep\in\Ends(\mc L^u)$ such that $\vep\neq\l_\infty$. Let $\gamma^\vep_j$ be the zigzag ray from $\l'_j$ to $\vep$. Then 
	\[
	\l_\infty\in\limset_{k\to\infty}\gamma^\vep_k.
	\]
\end{lemma}

\begin{proof}
	Let $U$ be a neighborhood of $\l_\infty$. We need to show that there exists some $K_U\in\N$ such that for all $k\geq K_U$, $\gamma^\vep_k\cap U\neq\varnothing$. The first thing to notice is that for all $k$,
	\[
	\gamma_k^\vep=\l'_k\cup\{\l^u\in\mc L^u\;\mid\;\ell^u\text{ separates }\ell'_k\text{ and }f(\vep)\text{ in }\overline{\orb}\}.
	\]
	
	We split the proof into the following three cases, either 
	\begin{enumerate}
		\item $z_\infty\in\orb$,
		\item $z_\infty\in\partial\orb^u$, or
		\item $z_\infty\in\partial\orb\ssm\partial\orb^u$.
	\end{enumerate}
	
	\begin{figure}[h]
		\centering
		\includegraphics[width=0.7\linewidth]{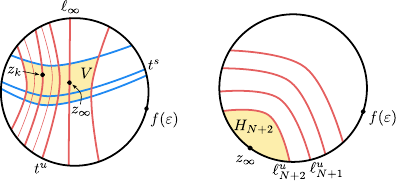}
		\caption{On the left, case (1) in the proof of \Cref{lem:linf_ray}, with $V$ shaded in yellow. On the right is case (3), with $H_{N+2}$ shaded in yellow.}\label{fig:linf_ray}
	\end{figure}

	In case (1), let $V$ be an open foliation chart around $z_\infty$ such that the projection of $V$ to $\mc L^u$ is contained in $U$. Since 
	\[
	\lim_{k\to\infty}z_k'=z_\infty,
	\]
	past some $K_V\in\N$ we have $z_k'\in V$, so $\ell'_k$ intersects $V$. Then using a stable arc $t^s$ in $V$, we may shift $\ell'_k$ slightly to one side to produce an unstable leaf $t^u$ separating $\ell'_k$ from $f(\vep)$ and intersecting $V$. See \Cref{fig:linf_ray}. Thus, $K_U=K_V$ works.
	
	For case (2), let $\ell_1^u,\ell^s_1,\ldots,\ell^u_n,\ell^s_n$ be the leaves of $\orb^{s/u}$ ending at the point $z_\infty$, and arranged in order as in \Cref{fig:linf_ray2} (it might be the case that $\ell_1^u$ or $\ell^s_n$ does not exist). A neighborhood $V$ for $z_\infty$ in $\overline{\orb}$ is given by the set bound by an arc of $\partial\orb$, together with segments of leaves $t^s_1,t^u_s,\ldots,t^s_n,t^u_n$ as in \Cref{fig:linf_ray2}. We can shrink this neighborhood to a neighborhood $V_U$ such that the projection of $V_U$ to $\mc L^u$ is contained in $U$ by shrinking the $t^{s/u}_j$'s along the $\ell^{s/u}_j$'s and by removing halfspaces bound by unstable leaves. We further shrink $V_U$ so that $f(\vep)\notin V_U$.

	Since 
	\[
	\lim_{k\to\infty}z_k'=z_\infty,
	\]
	we must have that past some $K_V\in\N$, $z_k'\in V_U$.  Either 
	\[
	z_k'\in\overline{\mc S^u(t_1^s\cup\ldots\cup t^s_n)},
	\]
	or $z_k'$ is separated from $f(\vep)$ by a leaf $\ell^u$ nonseparated from a leaf in $\mc S^u(t_1^s\cup\ldots\cup t^s_n)$. In the first case, we again can slightly nudge the unstable leaf $z_k'$ is on along the appropriate $t^s_j$ and toward $f(\vep)$ to produce a leaf $\ell^u$ in $\gamma_k^\vep\cap U$. In the second case, $\ell^u\in\gamma_k^\vep\cap U$.
	
		\begin{figure}[h]
		\centering
		\includegraphics[width=0.45\linewidth]{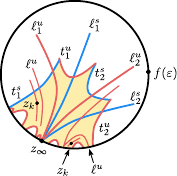}
		\caption{Case (2) for the proof of \Cref{lem:linf_ray}. The set $V_U$ is shaded in yellow. Two options for $z_k$ and the corresponding choices of $\ell^u$ are shown.}\label{fig:linf_ray2}
	\end{figure}
	
	For case (3), a neighborhood basis for $z_\infty$ in $\overline{\orb}$ is given by the half-planes $H_n$ bound by a set of leaves 
	\[
	\{\ell^u_n\;\mid\;n\in\N,\ell^u_n\in\orb^u\},
	\] as in \Cref{fig:linf_ray}. Let $V$ be one of these halfspaces. As in case (2), we shrink $V$ to an open set $V_U$ whose projection to $\mc L^u$ is contained in $U$ by picking a smaller halfspace $H_N$ in the above set and removing some halfspaces bound by unstable leaves. As in the previous cases, the $z_k'$'s eventually lie in $H_{N+2}$, so $\ell^u_{N+1}$ is a leaf in $\gamma_k^\vep\cap U$.

\end{proof}

Now we show that $s$ is ql-extremal over $\gamma$. We handle $\text{int}\;\gamma$ and $\{\l,\l_\infty\}$ separately.

\begin{lemma}\label{lem:ql_int}
	The section $s$ is ql-extremal over 
	$\gamma\ssm\{\l,\l_\infty\}$.
\end{lemma}

\begin{proof}
	Pick an end $\vep\in\Ends(\mc L^u)$ such that the zigzag ray $\gamma_\infty^\vep$ from $\l_\infty$ to $\vep$ passes through $\l$, so $\overline{\gamma}$ is an initial subpath of $\gamma_\infty^\vep$. Let $\mu\in\gamma\ssm\{\l,\l_\infty\}$, and suppose that $\gamma$ is travelling down on the maximal subinterval containing $\mu$ -- the case where $\gamma$ is travelling up is completely analogous. Denote this subinterval by $[\nu_0,\nu_1]$, allowing for the possibility that $\l_\infty=\nu_0$ or $\l=\nu_1$.
	
	By \Cref{lem:linf_ray}, there exists some $k\in\N$  such that 
	\[
	[\mu,\nu_1]\subseteq\bigcap_{j\geq k}\gamma_j^\vep,
	\]
	using the notation in \Cref{lem:linf_ray}. Then for all $j\geq k$,
	\[
	\mu\in\lu(s'_j),
	\]
	as $\gamma_j^\vep$ is a zigzag ray emanating from $\l'_j\in B(s'_j)$ and travelling up as it passes $\mu$.  However, 
	\[
	\mu\in\rd(s),
	\]
	as $\gamma$ is a zigzag ray from $\l\in B(s)$ that is travelling down as it passes $\mu$. See \Cref{fig:ql_int}. Thus, by \Cref{prop:ql_lim}, we have that $s$ is ql-extremal over a neighborhood of $\mu$ inside $\gamma$.
\end{proof}

\begin{figure}[h!]
	\centering
	\includegraphics[width=0.6\linewidth]{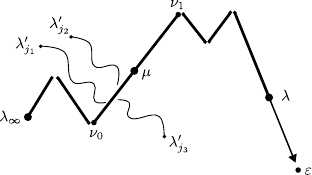}
	\caption{The setup in the proof of \Cref{lem:ql_int}. The subpath $\gamma$ of $\gamma^\vep_\infty$ is in bold. The fainter, squiggly paths represent the portions of the $\gamma_j^\vep$'s before they join up with $\gamma_\infty^\vep$.}\label{fig:ql_int}
\end{figure}

\begin{lemma}\label{lem:discrete_snm}
	The set 
	\[
	s\mid_\gamma\cap\snm
	\]
	is finite.
\end{lemma}

\begin{proof}
	This follows from \Cref{lem:qle_sprig} and \Cref{lem:ql_int}. As $s$ is ql-extremal over $\gamma\ssm\{\l,\l_\infty\}$, we have that 
	\[
	\Phi^{-1}(s\mid_{\gamma\ssm\{\l,\l_\infty\}})
	\]
	is contained in a master set. As $M$ is hyperbolic, master sets have finitely many leaves. Every nonmarker point in $s\mid_\gamma\cap\snm$ comes from an endpoint of a leaf in the master set, so this set is finite.
\end{proof}

\begin{lemma}\label{lem:ql_l_linf}
	The section $s$ is ql-extremal over $\gamma$.

\end{lemma}

\begin{proof}
	First we show that either $\l\in\Ends(\mc L^u)$, or $s$ is ql-extremal over $\gamma\ssm{\l_\infty}$. Of course, by \Cref{lem:ql_int}, we only need to worry about what is happening at and around $\l$. Suppose $\l\notin\Ends(\mc L^u)$. Then $s$ is a special section by \Cref{lem:no_i}, so $\l\in\lu(s)\cap\rd(s)$. By \Cref{lem:discrete_snm}, we know that $\snm\cap s\mid_\gamma$ is discrete, so $s$ is ql-extremal over $\gamma\ssm\{\l_\infty\}$.
	
	Now we show that either $\l_\infty\in\Ends(\mc L^u)$, or $s$ is ql-extremal over $\gamma\ssm\{\l\}$. As before, we are only concerned with what is happening around $\l_\infty$. Suppose $\l_\infty\notin\Ends(\mc L^u)$. Suppose further that $\gamma$ travels down as it approaches $\l_\infty$; the case where it travels up is analogous. If $s(\l_\infty)$ is a marker point, there's nothing to check, so suppose $s(\l_\infty)=\snm(\l_\infty)$. 
	
	By \Cref{lem:discrete_snm}, we know that over $\gamma$, $\snm(\l_\infty)$ is the lower endpoint of a marker $m$ in $s$. Our goal is to show $m$ is ql-extremal over $\gamma$.

	Suppose toward contradiction that $m$ is not ql-extremal over $\gamma$. Then $\Psi(m)$ is a stable leaf $m^s$ \emph{not} making a perfect fit with $\ell_\infty$ or with a leaf on an end of the branching chain containing $\ell_\infty$. See \Cref{fig:ql_l_linf}.
	
	\begin{figure}[h!]
		\centering
		\includegraphics[width=0.7\linewidth]{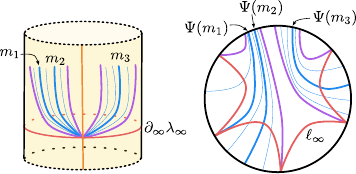}
		\caption{On the left, a possible arrangement of pinching leaves in the upper left and upper right quadrants of $\l_\infty$. Three possibilities for $m$ are $m_1$, $m_2$, and $m_3$ in bold. The thin blue markers are also possibilities, while the purple markers are not. On the right, the translation of this picture to $\orb$. Note that the picture on the left is not enough to decide how many leaves of $\orb^u$ in the branching chain lie to either side of $\ell_\infty$.}\label{fig:ql_l_linf}
	\end{figure}
	
	Since $s(\l_\infty)=\snm(\l_\infty)$, we have that $z_\infty\in\partial\orb$. Furthermore, we have by \Cref{lem:partial_vep} and the definition of $\wh\Psi$ that $z_\infty\in\partial\orb^u$ since $\l_\infty\notin\Ends(\mc L^u)$. Let $V$ be a small neighborhood of $z_\infty$ in $\overline{\orb}$ as in case (2) of the proof of \Cref{lem:linf_ray} such that $m^s\cap V=\varnothing$. Keeping the notation from the proof of \Cref{lem:linf_ray}, let $t^s_1,\ldots, t^s_n$ be the stable arcs on the frontier of $V$.
	
	Since 
	\[
	\lim_{k\to\infty}z_k'=z_\infty,
	\]
	we have that eventually the $z_k'$'s lie in $V$. If infinitely many of the $z_k'$'s lie in $\mc S^u(t^s_1\cup\cdots\cup t^s_n)$, then $\Phi(m^s)$ cannot be in $s$. This is because on a circle fiber $\partial_\infty\l_{N}$ just above $\partial_\infty\l_\infty$, we see that $s'_k(\l_N)$ stays away from $m$ as $k$ gets large.
	
	So, suppose only finitely many of the $z_k'$'s lie in $\mc S^u(t^2_1\cup\cdots\cup t^2_n)$, and delete all such points from the sequence. Since the $z_k'$'s eventually lie in $V$, this means the corresponding leaves $\l'_k$ are eventually incomparable with $\l_\infty$. As in the setup of \Cref{lem:linf_ray}, pick an end $\vep$ such that the zigzag ray $\gamma_\infty^\vep$ from $\l_\infty$ to $\vep$ contains $\gamma$ as an initial subpath. For all $k$, let $\gamma_k^\vep$ denote the zigzag ray from $\l'_k$ to $\vep$. Let $U\subset\mc L^u$ be an interval neighborhood of $\l_\infty$ intersecting $\gamma$ on a nondegenerate interval. Since $\gamma^\vep_k$ and $\gamma^\vep_\infty$ are zigzag rays to $\vep$, their intersection is a subray of each with an initial point $\nu_k\in\mc L^u$. By \Cref{lem:linf_ray}, we have 
	\[
	\l_\infty\in\limset_{k\to\infty}\nu_k,
	\]
	so in particular the $\nu_k$'s eventually lie in $U$. By the turning corners rule, $s'_k(\nu_k)=\snm(\nu_k)$, which prevents the $s'_k$'s from limiting onto an interior pinching marker with lower endpoint $\snm(\l_\infty)$ over $\gamma$.
\end{proof}

We can now apply \Cref{cor:ql_i} to show that $\ct$ is continuous.

\begin{proposition}\label{prop:ct_cts}
	The map 
	\[
	\ct\colon\mc C^\ell\to S^2_\infty
	\] 
	is continuous.
\end{proposition}

\begin{proof}
	
	By \Cref{lem:inf_lim},
	\[
	\lim_{j\to\infty}i\circ s'_j(\l'_j)=i\circ s(\l_\infty).
	\]
	Thus, we need to show that 
	\[
	i\circ s(\l)=i\circ s(\l_\infty).
	\]
	
	The section $s$ is ql-extremal over $\gamma$ by \Cref{lem:ql_int} and \Cref{lem:ql_l_linf}. Thus, if $\l_\infty,\l\notin\Ends(\mc L^u)$, then we know 
	\[
	i\circ s(\l)=i\circ s(\l_\infty)
	\]
	by \Cref{cor:ql_i}, and we are done. 
	
	Suppose instead that one or both of $\l$ or $\l_\infty$ is in $\Ends(\mc L^u)$, and denote it by $\l'$. Let $\{\mu_n\}_{n\in\N}$ be a sequence of leaves in $\gamma$ approaching $\l'$. Then 
	\[
	\lim_{n\to\infty}i\circ s(\mu_n)=i\circ s(\l')
	\]
	by \Cref{lem:i_hat_cts}. Since the terms in the limit are constant by \Cref{lem:ql_int}, we conclude that for all $\mu\in\gamma\ssm\{\l,\l_\infty\}$, we have
	\[
	i\circ s(\l')=i\circ s(\mu).
	\]
	Thus,
	\[
	i\circ s(\l)=i\circ s(\l_\infty).
	\]

 In every case we have shown
	\[
	\lim_{j\to\infty}\ct(s_j')=\ct(s).
	\]
	Since 
	\[
	\lim_{j\to\infty}\ct(s'_j)=\lim_{j\to\infty}\ct(s_j),
	\]
	and $\{s_j\}_{j\in\N}$ was any convergent subsequence, we have shown continuity.
\end{proof}

We now have everything we need to prove \Cref{thm:ct}.

\begin{proof}[Proof of \Cref{thm:ct}]
	By \Cref{prop:ct_wd}, 
	\[
	\ct\colon\mc C^\ell\to S^2_\infty
	\] 
	is a well-defined function, and it is continuous by \Cref{prop:ct_cts}. 
	
	The map $\ct$ is $\pi_1(M)$-equivariant since the maps $B$ and $i$ are. As usual, surjectivity then follows from the fact that every orbit of the action $\pi_1(M)\curvearrowright S^2_\infty$ is dense. 
	
	The last thing to check is that for all $s\in\mc C^\ell$,
	\[
	\ct(s)=i_\l\circ s(\l)
	\]
	whenever $s\in\core\pi_\l$. For any $s\in\mc C^\ell$ and leaf $\l\in\mc L^u$, the section $s'=s^\ell_{s(\l)}$ corresponds to a point in $\mc C^\ell$ mapped to the same place as $s$ under $\pi_\l$. Thus, if $s\in\core\pi_\l$, we must have $s'=s$. Then by definition,
	\[
	\ct(s)=i_\l\circ s(\l),
	\]
	so $\ct$ is a Cannon--Thurston map for $W^u$ in the sense of \Cref{def:ct_for_fol}.
\end{proof}

\subsection{Dynamics on the leftmost circle}
Finally, we apply \Cref{thm:ct} to obtain information about the dynamics of the action of $\pi_1(M)$ on $\mc C^\ell$. In particular, we show the following.

\begin{corollary}\label{corollary:dynamics}
	Every element of $\pi_1(M)$ has some power acting on $\mc C^\ell$ with a positive, finite number of fixed points, which alternate between attractors and repellors.
\end{corollary}

Note that the action of $\pi_1(M)$ on $\partial\orb$ also has this property. Thus, \Cref{corollary:dynamics} can be seen as further evidence that $\pi_1(M)\curvearrowright\mc C^\ell$ should be conjugate to an action of $\pi_1(M)$ on some orbit space boundary $\partial\orb_\psi$ for $\psi$ a pseudo-Anosov flow transverse to $W^u$.

\begin{proof}[Proof of \Cref{corollary:dynamics}]
	Let $g\in\pi_1(M)$. Since $M$ is a closed hyperbolic 3-manifold, we know $g$ acts on $\wt M\cong\HH^3$ as a hyperbolic isometry. In particular, $g$ fixes exactly two points on $S^2_\infty$, which are a sink and a source. Denote these points by $g_+$ and $g_-$ respectively.
	
	We first show that the set of points of $\mc C^\ell$ that are periodic under $g$ are exactly the points sent to $g_\pm$ by $\ct$. First suppose $s\in\mc C^\ell$ is periodic under $g$. Then by $\pi_1(M)$-equivariance, $\ct(s)$ is also periodic under $g$, and thus $\ct(s)=g_\pm$. For the other direction, suppose $\ct(s)=g_\pm$. Then by $\pi_1(M)$-euivariance, the entire $g$-orbit of $s$ is sent to $g_\pm$. Since every Cannon--Thurston map is uniformly finite-to-one by \cite{FMP}, we have that the $g$-orbit of $s$ is finite, and therefore $s$ is periodic under $g$.
	
	Since $\ct$ is surjective and uniformly finite-to-one, there are a positve, finite number of points of $\mc C^\ell$ that are periodic under $g$. Fix a positive power $k$ of $g$ such that $g^k$ fixes each of these points. Let $s$ be a point in $\mc C^\ell$ that is sent to $g_+$ under $\ct$. By continuity of $\ct$, some small neighborhood of $s$ in $\mc C^\ell$ is mapped into a small neighborhood of $g_+$ in $S^2_\infty$, and so $s$ is a local sink for the action of $g^k$ on $\mc C^\ell$. Similarly, points of $\mc C^\ell$ sent to $g_-$ are local sources. Thus, we must have that the points of $\ct^{-1}(g_-)$ and the points of $\ct^{-1}(g_+)$ are arranged around $\mc C^\ell$ in an alternating fashion, and in the complementary intervals, points are moved from the local source to the local sink.
\end{proof}

\bibliographystyle{alpha}
\bibliography{anosov_ct.bib}

\end{document}